\documentclass{svjour3}

\usepackage{amsmath,amssymb,amsfonts,latexsym}
\usepackage{mathrsfs}
\usepackage{graphicx,float,epsfig}
\usepackage[dvips]{color}
\smartqed

\journalname{Numerische Mathematik}

\def\C{{\mathbb{C}}}
\def\CC{\mathcal{C}}
\def\CT{\mathcal{T}}
\def\G{\Gamma}
\def\HrE{{\H^{r}(E)}}
\def\HsO{{\H^{s}(\O)}}
\def\HtE{{\H^{t}(E)}}
\def\HtO{{\H^{t}(\O)}}
\def\HuE{{\H^1(E)}}
\def\HuO{{\H^1(\O)}}
\def\HusO{{\H^{1+s}(\O)}}
\def\HutO{{\H^{1+t}(\O)}}
\def\H{\mathrm{H}}
\def\L{\mathrm{L}}
\def\LO{\L^2(\O)}
\def\N{{\mathbb{N}}}
\def\O{\Omega}
\def\R{{\mathbb{R}}}
\def\T{{\mathcal T}}
\def\bE{\boldsymbol{\mathcal{E}}}
\def\bF{\mathbf{F}}
\def\bG{\boldsymbol{\mathcal{G}}}
\def\bI{\mathbf{I}}
\def\bK{\boldsymbol{\mathcal{K}}}
\def\bM{\mathbf{M}}
\def\bP{\mathbf{P}}
\def\bS{\mathbf{S}}
\def\bT{\mathbf{T}}
\def\bV{\boldsymbol{\mathcal{V}}}
\def\bX{\boldsymbol{\mathcal{X}}}
\def\bY{\boldsymbol{\mathcal{Y}}}
\def\bf{\boldsymbol{f}}
\def\bn{\boldsymbol{n}}
\def\bphi{\boldsymbol{\varphi}}
\def\bpi{\boldsymbol{\Pi}}
\def\bpsi{\boldsymbol{\psi}}
\def\bu{\boldsymbol{u}}
\def\bv{\boldsymbol{v}}
\def\bw{\boldsymbol{w}}
\def\bx{\boldsymbol{x}}
\def\by{\boldsymbol{y}}
\def\dim{\mathop{\mathrm{dim}}\nolimits}
\def\div{\mathop{\mathrm{div}}\nolimits}
\def\divO{\div,\O}
\def\ds{\,ds}

\def\hdel{\widehat{\delta}}
\def\l{\lambda}

\def\rot{\mathop{\mathrm{rot}}\nolimits}

\newcommand\HdivO{{\H(\div;\O)}}
\newcommand\0{\boldsymbol{0}}
\renewcommand\sp{\mathop{\mathrm{sp}}\nolimits}
\newcommand\bbP{\mathbb{P}}

\begin{document}

\title{A virtual element method for the acoustic vibration problem}
\author{Louren\c{c}o Beir\~ao da Veiga 
\and David Mora\thanks{D. Mora was partially supported by CONICYT
(Chile) through FONDECYT project No. 1140791, by DIUBB through project
151408 GI/VC and by Anillo ANANUM, ACT1118, CONICYT (Chile).} 
\and Gonzalo Rivera\thanks{G. Rivera was partially supported by a
CONICYT (Chile) fellowship.}
\and Rodolfo Rodr\'iguez\thanks{R. Rodr\'iguez was partially supported
by BASAL project, CMM, Universidad de Chile, by Anillo ANANUM, ACT1118,
CONICYT (Chile) and by Red Doctoral REDOC.CTA, MINEDUC project UCO1202
at Universidad de Concepci\'on (Chile).}}
\institute{L. Beir\~ao da Veiga 
\at Dipartimento di Matematica e Applicazioni, Universit\`a di
Milano-Bicocca, 20125 Milano, Italy.\\
\email{lourenco.beirao@unimib.it} 
\and D. Mora 
\at GIMNAP, Departamento de Matem\'atica, Universidad del B\'io-B\'io,
Casilla 5-C, Concepci\'on, Chile and Centro de Investigaci\'on en
Ingenier\'\i a Matem\'atica (CI$^2$MA), Universidad de Concepci\'on,
Casilla 160-C, Concepci\'on, Chile.\\
\email{dmora@ubiobio.cl} 
\and G. Rivera \and R. Rodr\'iguez 
\at CI$^2$MA, Departamento de Ingenier\'{\i}a Matem\'atica, Universidad
de Concepci\'on, Casilla 160-C, Concepci\'on, Chile.\\
\email{\{grivera,rodolfo\}@ing-mat.udec.cl}}

\maketitle

\begin{abstract}
We analyze in this paper a virtual element approximation for the
acoustic vibration problem. We consider a variational formulation
relying only on the fluid displacement and propose a discretization by
means of $\H(\div)$ virtual elements with vanishing rotor. Under
standard assumptions on the meshes, we show that the resulting scheme
provides a correct approximation of the spectrum and prove optimal order
error estimates. With this end, we prove approximation properties of the
proposed virtual elements. We also report some numerical tests
supporting our theoretical results.
\end{abstract}



\section{Introduction}
\label{SEC:INTR}

The {\it Virtual Element Method} (VEM) introduced in \cite{BBCMMR2013}
is a recent generalization of the Finite Element Method which is
characterized by the capability of dealing with very general
polygonal/polyhedral meshes and the possibility to easily implement
highly regular discrete spaces. Indeed, by avoiding the explicit
construction of the local basis functions, the VEM can easily handle
general polygons/polyhedrons without complex integrations on the element
(see \cite{BBMR2014} for details on the coding aspects of the method).
The interest in numerical methods that can make use of general polytopal
meshes has recently undergone a significant growth in the mathematical
and engineering literature; among the large number of papers on this
subject, we cite as a minimal sample
\cite{AHSV,BLMbook2014,CGH14,DPECMAME2015,DPECRAS2015,ST04,TPPM10}.
Regarding the VEM literature, we limit ourselves to the following few
articles
\cite{equiv,ABMV2014,ALM15,BBCMMR2013,BBMR2014,BBMR2015,BLM2015,BBPS2014,ultimo,BM12,Paulino-VEM,MRR2015,PPR15}.

The numerical approximation of eigenvalue problems for partial
differential equations derived from engineering applications, is object
of great interest from both, the practical and theoretical points of
view. We refer to \cite{Boffi,BGG2012} and the references therein for
the state of the art in this subject area. In particular, this paper
focus on the so called acoustic vibration problem; namely, to compute
the vibration modes and the natural frequencies of an inviscid
compressible fluid within a rigid cavity \cite{I98}. One motivation for
considering this problem is that it constitutes a stepping stone towards
the more challenging goal of devising virtual element spectral
approximations for coupled systems involving fluid-structure
interaction, which arises in many engineering problems. The simplest
formulation of this problem is obtained by using pressure variations
which leads to an eigenvalue problem for the Laplace
operator~\cite{I98}. However, for coupled problems, it is convenient to
use a dual formulation in terms of fluid displacements (see
\cite{EM77}). A standard finite element approximation of this problem
leads to spurious modes (see \cite{DI2001}). Such a spectral pollution
can be avoided by using $\H(\div)$-conforming elements, like
Raviart-Thomas finite elements
\cite{BDMR95,B96,BHPR2001,BGG2012,HST2000}. See \cite{BGHRS2008} for a
thorough discussion on this topic.

The aim of this paper is to introduce and analyze an $\H(\div)$ VEM
which applies to general polygonal (even non-convex) meshes for the
two-dimensional acoustic vibration problem. We begin with a variational
formulation of the spectral problem relying only on the fluid
displacement. Then, we propose a discretization based on the mixed VEM
introduced in \cite{BBMR2015} for general second order elliptic
problems. The well-known abstract spectral approximation theory (see
\cite{BO}) cannot be used to deal with the analysis of our problem.
Indeed, the kernel of the bilinear form on the left-hand side of the
variational formulation has in our case an infinite-dimensional kernel.
Although the standard shift strategy allows a solution operator to be
defined, this is not compact and its nontrivial essential spectrum may
in such cases lead to spectral pollution at the discrete level. 
However, by appropriately adapting the abstract spectral approximation
theory for non-compact operators developed in \cite{Raviart1,Raviart2},
under rather mild assumptions on the polygonal meshes, we establish that
the resulting scheme provides a correct approximation of the spectrum
and prove error estimates for the eigenfunctions and a double order for
the eigenvalues. As a by-product, we derive optimal approximation
estimates for $\H(\div)$ virtual elements with vanishing rotor, a result
that could be useful also for other applications. These results and
their corresponding proofs are collected in an appendix.

The outline of this article is as follows: We introduce in
Section~\ref{SEC:STAT} the variational formulation of the acoustic
vibration problem, define a solution operator and establish its spectral
characterization. In Section~\ref{SEC:Discrete}, we introduce the
virtual element discrete formulation, describe the spectrum of a
discrete solution operator and establish some auxiliary results. In
Section~\ref{SEC:approximation}, we prove that the numerical scheme
provides a correct spectral approximation and establish optimal order
error estimates for the eigenvalues and eigenfunctions. In
Section~\ref{SEC:NUMER}, we report a couple of numerical tests that
allow us to assess the convergence properties of the method, to confirm
that it is not polluted with spurious modes and to check that the
experimental rates of convergence agree with the theoretical ones.
Finally, we introduce in an appendix the proofs of the approximation
results for the introduced virtual element interpolant.

Throughout the paper, $\O$ is a generic Lipschitz bounded domain of
$\R^2$. For $s\geq 0$, $\left\|\cdot\right\|_{s,\O}$ stands indistinctly
for the norm of the Hilbertian Sobolev spaces $\HsO$ or $[\HsO]^2$ with
the convention $\H^0(\O):=\LO$. We also define the Hilbert space
$\HdivO:=\left\{\bv\in[\LO]^2:\ \div\bv\in\LO\right\}$, whose norm is
given by $\left\|\bv\right\|^2_{\divO}
:=\left\|\bv\right\|_{0,\O}^2+\left\|\div\bv\right\|^2_{0,\O}$. Finally,
we employ $\0$ to denote a generic null vector and $C$ to denote generic
constants independent of the discretization parameters, which may take
different values at different places.

\section{The spectral problem}
\label{SEC:STAT}

We consider the free vibration problem for an acoustic fluid within a
bounded rigid cavity $\O\subset\R^2$ with polygonal boundary $\G$ and 
outward unit normal vector $\bn$:

\begin{equation*}
\left\{\begin{array}{ll}
\vspace{0.1cm}
-\omega^2\varrho\bw=-\nabla p\quad & \text{in }\O,
\\
\vspace{0.1cm}
p=-\varrho c^2\div\bw\quad & \text{in }\O,
\\
\bw\cdot\bn=0\quad & \text{on }\G,
\end{array}\right.
\label{vibration}
\end{equation*}
where $\bw$ is the fluid displacement, $p$ is the pressure fluctuation,
$\varrho$ the density, $c$ the acoustic speed and $\omega$ the vibration
frequency. Multiplying the first equation above by a test function
$$
\bv\in\bV:=\left\{\bv\in\HdivO:\ \bv\cdot\bn=0\text{ on }\G\right\},
$$ 
integrating by parts, using the boundary condition and eliminating $p$,
we arrive at the following weak formulation in which, for simplicity, we
have taken the physical parameters $\varrho$ and $c$ equal to one and
denote $\l=\omega^2$:

\begin{problem}
\label{P1}
Find $(\l,\bw)\in\R\times\bV$, $\bw\neq 0$, such that 
\begin{equation*}
\label{2}
\int_{\O}\div\bw\div\bv 
=\l\int_{\O}\bw\cdot\bv
\qquad\forall\bv\in\bV.
\end{equation*}
\end{problem}

Since the bilinear form on the left-hand side is not $\HdivO$-elliptic,
it is convenient to use a shift argument to rewrite this eigenvalue
problem in the following equivalent form: 

\begin{problem}
\label{P2}
Find $(\l,\bw)\in\R\times\bV$, $\bw\neq 0$, such that
\begin{equation*}
\label{3}
a(\bw,\bv)=\left(\l+1\right)b(\bw,\bv)\qquad\forall\bv\in\bV,
\end{equation*}
\end{problem}
where the bilinear forms are defined for any $\bw,\bv\in\bV$ by
\begin{align*}
a(\bw,\bv) & :=\int_{\O}\div\bw\div\bv+\int_{\O}\bw\cdot\bv,
\\
b(\bw,\bv) & :=\int_{\O}\bw\cdot\bv.
\end{align*}

We define the solution operator associated with Problem~\ref{P2}:
\begin{align*}
\bT:\;\bV & \longrightarrow\bV,
\\
\bf & \longmapsto\bT\bf:=\bu,
\end{align*}
where $\bu\in\bV$ is the solution of the corresponding source problem:
\begin{equation*}\label{4}
a(\bu,\bv)=b(\bf,\bv)\qquad\forall\bv\in\bV.
\end{equation*}
Since the bilinear form $a(\cdot,\cdot)$ is $\HdivO$-elliptic, the
problem above is well posed. As an immediate consequence, we deduce that
the linear operator $\bT$ is well defined and bounded. Notice that
$(\l,\bw)\in\R\times\bV$ solves Problem~\ref{P1} if and only if
$(1/\left(1+\l\right),\bw)$ is an eigenpair of $\bT$, i.e, if and only
if 
$$
\bT\bw=\mu\bw,\quad\text{ with }\mu:=\dfrac{1}{\l+1}.
$$
Moreover, it is easy to check that $\bT$ is self-adjoint with respect to
the inner products $a(\cdot,\cdot)$ and $b(\cdot,\cdot)$ in $\bV$.

In what follows, we recall some results that can be found in
\cite{BDMR95} in the more general context of fluid-solid vibration
problems. The proofs in \cite{BDMR95} can be readily adapted to this
case to obtain the following results. Let the space
$$
\bK:=\left\{\bv\in\bV:\ \div\bv=0\ \text{in}\ \O\right\}.
$$

\begin{lemma}
\label{muigual1}
The operator $\bT$ admits the eigenvalue $\mu=1$ with associated
eigenspace $\bK$.
\end{lemma}

The following result provides a simple characterization of the
orthogonal complement of $\bK$ in $\bV$.

\begin{lemma}
\label{ortho}
Let $\bG:=\left\{\nabla q:\ q\in\HuO\right\}$. Then, 
$$
\bV=\bK\oplus\left(\bG\cap\bV\right),
$$
is an orthogonal decomposition in both $[\LO]^2$ and $\,\HdivO$. 

Moreover, there exists $s\in(1/2,1]$ such that, for all $\bv\in\bV$, if 
$\bv=\bphi+\nabla q$ with $\bphi\in\bK$ and $\nabla q\in\bG\cap\bV$,
then $\nabla q\in[\HsO]^2$ and $\left\|\nabla q\right\|_{s,\O}
\leq C\left\|\div\bv\right\|_{0,\O}$.
\end{lemma}

{}From now on, we fix $s\in (1/2,1]$ such that the above lemma holds
true.

The following result shows that the subspace $\bG\cap\bV$ is invariant
for $\bT$.

\begin{lemma}
\label{invariant}
There holds 
$$
\bT(\bG\cap\bV)\subset\left(\bG\cap\bV\right).
$$
\end{lemma}

Smoothing properties of $\bT$ as an operator from $\bG\cap\bV$ into
itself are established in what follows.

\begin{theorem}
\label{teo2.5}
There holds
$$
\bT(\bG\cap\bV)\subset\left\{\bv\in[\HsO]^2:\ \div\bv\in\HuO\right\}
$$
and there exists $C>0$ such that, for all $\bf\in\bG\cap\bV$, if
$\bu=\bT\bf$, then
\begin{equation*}
\left\|\bu\right\|_{s,\O}
+\left\|\div\bu\right\|_{1,\O}
\leq C\left\|\bf\right\|_{\divO}.
\end{equation*}
Consequently, the operator $\bT|_{\bG\cap\bV}:\ \bG\cap\bV\to\bG\cap\bV$
is compact.
\end{theorem}

Finally, the following result provides a spectral characterization of
$\bT$.

\begin{theorem}
\label{CHAR_SP}
The spectrum of $\bT$ decomposes as
$\sp(\bT)=\left\{0,1\right\}\cup\left\{\mu_k\right\}_{k\in\N}$, where:
\begin{enumerate}
\item[\textit{i)}] $\mu=1$ is an infinite-multiplicity eigenvalue of 
$\bT$ and its associated eigenspace is $\bK$;
\item[\textit{ii)}] $\left\{\mu_k\right\}_{k\in\N}\subset(0,1)$ is a
sequence of finite-multiplicity eigenvalues of $\bT$ which converge to 
$0$ and if $\bw$ is an eigenfunction of $\bT$ associated with such an
eigenvalue, then there exists $\tilde{s}>1/2$ and $C>0$, both depending
on the eigenvalue, such that
$$
\left\|\bw\right\|_{\tilde{s},\O}
+\left\|\div\bw\right\|_{1+\tilde{s},\O}
\leq C\left\|\bw\right\|_{\divO};
$$
\item[\textit{iii)}] $\mu=0$ is not an eigenvalue of $\bT$.
\end{enumerate}
\end{theorem}

\section{The virtual elements discretization}
\label{SEC:Discrete}

We begin this section, by recalling the mesh construction and the
assumptions considered to introduce a discrete virtual element space.
Then, we will introduce a virtual element discretization of
Problem~\ref{P1} and provide a spectral characterization of the
resulting discrete eigenvalue problem. Let $\left\{\CT_h\right\}$ be a
family of decompositions of $\O$ into polygons $E$. Let $h_E$ denote
the diameter of the element $E$ and $h:=\max_{E\in\O}h_E$.

For the analysis, we make the following assumptions on the meshes as in
\cite{BBMR2015,ultimo}: there exists a positive real number $C_{\T}$
such that, for every $E\in\T_h$ and for every $\CT_h$,
\begin{itemize}
\item $\mathbf{A_1}$: the ratio between the shortest edge and the
diameter of $E$ is larger than $C_{\T}$;
\\[-.3cm]
\item $\mathbf{A_2}$: $E$ is star-shaped with respect to every
point of a ball of radius $C_{\T}h_E$.
\end{itemize}

For any subset $S\subseteq\R^2$ and any non-negative integer $k$, we
indicate by $\bbP_k(S)$ the space of polynomials of degree up to $k$
defined on $S$. To keep the notation simpler, we denote by $\bn$ a
generic normal unit vector; in each case, its precise definition will be
clear from the context. We consider now a polygon $E$ and, for any fixed
non-negative integer $k$, we define the following finite dimensional
space (inspired in \cite{ultimo,BBMR2015}):
$$
\bV_{h}^{E}:=\Big\{\bv_h\in\H(\div;E):
\ \left(\bv_h\cdot\bn\right)\in\bbP_k(e)
\ \,\forall e\subset\partial E,\ \div\bv_h\in\bbP_k(E),
\ \rot\bv_h=0\text{ in }E\Big\}.
$$

\begin{remark}
\label{obs1}
It is elementary to check that a vector field $\bv_h\in\bV_h^E$
satisfying $\bv_h\cdot\bn=0$ on $\partial E$ and $\div\bv_h=0$ in $E$
is identically zero. In fact, since a star-shaped polygon $E$ is simply
connected and $\rot\bv_h=0$ in $E$, there exists $\gamma\in\HuE$ such
that $\bv_h=\nabla\gamma$. Then, $\Delta\gamma=\div\bv_h=0$ in $E$ and
$\partial\gamma/\partial\bn=\bv_h\cdot\bn=0$ on $\partial E$. Hence,
$\bv_h=\nabla\gamma=\0$ in $E$. This implies that $\bV_h^E$ is finite
dimensional, the dimension being less or equal to
$N_E\left(k+1\right)+\left(k+1\right)\left(k+2\right)/2-1$, where $N_E$
is the number of edges of $E$.
\end{remark}

We define the following degrees of freedom for functions $\bv_h$ in
$\bV_h^E$:
\begin{align}
\label{freedom}
\int_{e}\left(\bv_h\cdot\bn\right)q\ds
& \qquad\forall q\in\bbP_k(e),
\quad\forall\text{ edge }e\subset\partial E;
\\
\label{freedom2}
\int_{E}\bv_h\cdot\nabla q
& \qquad\forall q\in\bbP_k(E)/\bbP_{0}(E).
\end{align}

\begin{proposition}
The degrees of freedom \eqref{freedom}--\eqref{freedom2} are unisolvent
in $\bV_h^E$.
\end{proposition}

\begin{proof}
It is easy to check that the number of degrees of freedom
\eqref{freedom}--\eqref{freedom2} equals the dimension of $\bV_h^E$.
Thus, we only need to show that if $\bv_h$ in $\bV_h^E$ is such that
\begin{align*}
\int_{e}\left(\bv_h\cdot\bn\right)q\ds
& =0\qquad\forall q\in\bbP_k(e),
\quad\forall\text{ edge }e\subset\partial E,
\\
\int_{E}\bv_h\cdot\nabla q
& =0\qquad\forall q\in\bbP_k(E)/\bbP_{0}(E),
\end{align*}
then $\bv_h=\0$. Since $\div\bv_h\in\bbP_k(E)$, by taking
$q:=\div\bv_h$ above, we have
$$
\int_{E}\left(\div\bv_h\right)^2
=\int_{E}\div\bv_h\,q
=-\int_{E}\bv_h\cdot\nabla q
+\int_{\partial E}\left(\bv_h\cdot\bn\right) q\ds=0.
$$
Then, $\div\bv_{h}=0$. Similarly, for each edge $e\subset\partial E$,
since $\bv_h\cdot\bn\in\bbP_k(e)$, by taking $q:=\bv_h\cdot\bn$ we
obtain 
$$
\int_{e}\left(\bv_h\cdot\bn\right)^{2}\ds=0.
$$
Hence, $\bv_h\cdot\bn=0$ on $\partial E$. Therefore, according to
Remark~\ref{obs1}, $\bv_h=\0$ in $E$.
$\hfill\qed$
\end{proof}

\begin{remark}
\label{freedom3}
For the degrees of freedom \eqref{freedom2}, we could integrate by parts
and substitute them with
\begin{equation*}
\int_E\div\bv_h\;q\qquad\forall q\in\bbP_k(E)/\bbP_{0}(E).
\end{equation*}
Needless to say, certain degrees of freedom will be more convenient when
writing the code and the others might be more convenient when writing a
proof.
\end{remark}

For each decomposition $\CT_h$ of $\O$ into polygons $E$, we define
\begin{align}
\bV_h:=\left\{\bv_h\in\bV:\ \bv_h|_E\in\bV_h^E\right\}.
\nonumber
\end{align}
In agreement with the local choice, we choose the following global
degrees of freedom:
\begin{align*}
\label{globalfreedom}
\int_{e}\left(\bv_h\cdot\bn\right)q\ds
& \qquad\forall q\in\bbP_k(e),
\quad\text{for each internal edge }e\not\subset\G;
\\
\int_{E}\bv_h\cdot\nabla q
& \qquad\forall q\in\bbP_k(E)/\bbP_{0}(E),
\quad\text{for each element }E\in\CT_h.
\\
\end{align*}

\begin{remark}
The number of internal degrees of freedom of the Virtual Element Method
here considered ($VEM_k$) is in general less than that of standard
finite elements of the same order such as Raviart-Thomas ($RT_k$) or
Brezzi-Douglas-Marini ($BDM_k$) elements, while the number of degrees of
freedom per edge is the same. A count of the internal degrees of freedom
gives
$$
RT_k:\ k\left(k+1\right), 
\qquad BDM_k:\ \left(k+1\right)\left(k-1\right), 
\qquad VEM_k:\ \left(k+1\right)\left(k+2\right)/2-1.
$$
The proposed family may therefore be preferable to more standard finite
elements even in the case of triangular meshes, especially for
moderate-to-high values of $k$. 
\end{remark}

In order to construct the discrete scheme, we need some preliminary
definitions. First, we split the bilinear form $a(\cdot,\cdot)$
introduced in the previous section as follows:
$$
a(\bu_h,\bv_h)
=\sum_{E\in\CT_h}\left(\int_{E}\div\bu_h\div\bv_h
+\int_{E}\bu_h\cdot\bv_h\right),
\qquad\bu_h,\bv_h\in\bV_h.
$$
The local matrices associated with the first term on the right hand side
above are easily computable since $\div\bu_h$ and $\div\bv_h$ are
polynomials in each element. We explicitly point out that, as can be
seem from \eqref{freedom}--\eqref{freedom2}, the divergence of any
vector $\bv_h\in\bV_h$ can be easily computed from knowledge of the
degrees of freedom of $\bv_h$. Instead, for the local matrices
associated with the second term on the right hand side above, we must
take into account that, due to the implicit space definition, it is not
possible to compute exactly the integrals. Because of this, we will use
an approximation of them. The final output will be a local matrix on
each element $E$ whose associated bilinear form is exact whenever one of
the two entries is a gradient of a polynomial of degree $k+1$. This will
allow us to retain the optimal approximation properties of the space
$\bV_h$. With this aim, we define first for each element $E$ the space
\begin{equation*}
\label{Ve}
\widehat{\bV}_h^E
:=\nabla(\bbP_{k+1}(E))
\subset\bV_{h}^{E}.
\end{equation*}
Then, we define the $[\L^2(E)]^2$-orthogonal projector
$\bpi_h^E:\;[\L^2(E)]^2\longrightarrow\widehat{\bV}_h^E$ by
\begin{equation}
\label{numero}
\int_{E}\bpi_h^E\bv\cdot\widehat{\bu}_h
=\int_{E}\bv\cdot\widehat{\bu}_h
\qquad\forall\widehat{\bu}_h\in\widehat{\bV}_h^E.
\end{equation}
We point out that $\bpi_h^E\bv_h$ is explicitly computable for every
$\bv_h\in\bV_h^E$ using only its degrees of freedom
\eqref{freedom}--\eqref{freedom2}. In fact, it is easy to check that for
all $\bv_h\in\bV_h^E$ and for all $ q\in\bbP_{k+1}(E)$,
\begin{equation*}
\int_{E}\bpi_h^E\bv_h\cdot\nabla q
=\int_{E}\bv_h\cdot\nabla q
=-\int_{E}\div\bv_h\,q
+\int_{\partial E}\left(\bv_h\cdot\bn\right)q\ds.
\end{equation*}

\begin{remark}
In particular, for $k=0$, for all $\bv_h\in\bV_h^E$ and for all
$q\in\bbP_{1}(E)$, we have that
$$
\int_{E}\bpi_h^E\bv_h\cdot\nabla q
=-\left(\dfrac{1}{\left|E\right|}
\sum_{e\subset\partial E}\int_{e}\bv_h\cdot\bn\ds\right)
\left(\int_{E}q\right)
+\sum_{e\subset\partial E}\int_{e}\left(\bv_h\cdot\bn\right)q\ds.
$$
\end{remark}

On the other hand, let $S^E(\cdot,\cdot)$ be any symmetric positive
definite (and computable) bilinear form to be chosen as to satisfy
\begin{equation}
\label{20}
c_0\,\int_{E}\bv_h\cdot\bv_h\leq S^{E}(\bv_h,\bv_h)
\leq c_1\,\int_{E}\bv_h\cdot\bv_h
\qquad\forall\bv_h\in\bV_h^E,
\end{equation}
for some positive constants $c_0$ and $c_1$ depending only on the 
constant $C_{\T}$ from mesh assumptions $\mathbf{A_1}$ and
$\mathbf{A_2}$. Then, we define on each element $E$ the bilinear form
\begin{equation}
\label{21}
b_h^{E}(\bu_h,\bv_h)
:=\int_{E}\bpi_h^E\bu_h\cdot\bpi_h^E\bv_h
+S^{E}\big(\bu_h-\bpi_h^E\bu_h,\bv_h-\bpi_h^E\bv_h\big),
\qquad\bu_h,\bv_h\in\bV_h^E, 
\end{equation}
and, in a natural way,
$$
b_h(\bu_h,\bv_h)
:=\sum_{E\in\CT_h}b_h^{E}(\bu_h,\bv_h),
\qquad\bu_h,\bv_h\in\bV_h.
$$
The following two properties of the bilinear form $b_h^E(\cdot,\cdot)$
are easily derived by repeating in our case the arguments from
\cite[Proposition~4.1]{ultimo}.
\begin{itemize}
\item \textit{Consistency}: 
\begin{equation}
\label{consistencia1}
b_h^{E}(\widehat{\bu}_h,\bv_h)
=\int_{E}\widehat{\bu}_h\cdot\bv_h
\qquad\forall\widehat{\bu}_h\in\widehat{\bV}_h^E,
\quad\forall\bv_h\in\bV_h^E,\quad\forall E\in\CT_h.
\end{equation}
\item \textit{Stability}: There exist two positive constants $\alpha_*$
and $\alpha^*$, independent of $E$, such that:
\begin{equation}
\label{consistencia2}
\alpha_*\int_{E}\bv_h\cdot\bv_h
\leq b_h^{E}(\bv_h,\bv_h)
\leq\alpha^*\int_{E}\bv_h\cdot\bv_h
\qquad\forall\bv_h\in\bV_h^E,\quad\forall E\in\CT_h. 
\end{equation}
\end{itemize}

Now, we are in a position to write the virtual element discretization of
Problem~\ref{P1}.

\begin{problem}
\label{P3}
Find $(\l_h,\bw_h)\in\R\times\bV_h$, $\bw_h\neq 0$, such that 
\begin{equation*}
\label{vp}
\int_{\O}\div\bw_h\div\bv_h
=\l_h b_h(\bw_h,\bv_h)
\qquad\forall\bv_h\in\bV_h.
\end{equation*}
\end{problem}

We use again a shift argument to rewrite this discrete eigenvalue
problem in the following convenient equivalent form.

\begin{problem}
\label{P4}
Find $(\l_h,\bw_h)\in\R\times\bV_h$, $\bw_h\neq 0$, such that
\begin{equation*}
\label{vp*2}
a_h(\bw_h,\bv_h)
=\left(\l_h+1\right)b_h(\bw_h,\bv_h)
\qquad\forall\bv_h\in\bV_h,
\end{equation*}
\end{problem}
where
$$
a_h(\bw_h,\bv_h)
:=\int_{\O}\div\bw_h\div\bv_h
+b_h(\bw_h,\bv_h)
\qquad\forall\bw_h,\bv_h\in\bV_h.
$$
 
We observe that by virtue of \eqref{consistencia2}, the bilinear form
$a_h(\cdot,\cdot)$ is bounded. Moreover, as is shown in the following
lemma, it is also uniformly elliptic.

\begin{lemma}
\label{ha-elipt-disc}
There exists a constant $\beta>0$, independent of $h$, such that
$$
a_h(\bv_h,\bv_h)
\ge\beta\left\|\bv_h\right\|_{\divO}^2
\qquad\forall\bv_h\in\bV_h.
$$
\end{lemma}

\begin{proof}
Thanks to \eqref{consistencia2}, the above inequality holds with
$\beta:=\min\left\{\alpha_{*},1\right\}$.
$\hfill\qed$
\end{proof}

The next step is to introduce the discrete version of the operator
$\bT$:
\begin{align*}
\bT_h:\;\bV_h & \longrightarrow\bV_h,
\\
\bf_h & \longmapsto\bT_h\bf_h:=\bu_h,
\end{align*}
where $\bu_h\in\bV_h$ is the solution of the corresponding discrete
source problem:
\begin{equation*}
\label{T2}
a_h(\bu_h,\bv_h)=b_h(\bf_h,\bv_h)\qquad\forall\bv_h\in\bV_h.
\end{equation*}
We deduce from Lemma~\ref{ha-elipt-disc}, \eqref{consistencia2} and the
Lax-Milgram Theorem, that the linear operator $\bT_h$ is well defined
and bounded uniformly with respect to $h$.

Once more, as in the continuous case, $(\l_h,\bw_h)$ solves
Problem~\ref{P3} if and only if $(1/(1+\l_h),\bw_h)$ is an eigenpair of
$\bT_h$, i.e, if and only if 
$$
\bT_h\bw_h=\mu_h\bw_h,
\quad\text{ with }\mu_h:=\dfrac{1}{\l_h+1}.
$$
Moreover, it is easy to check that $\bT_h$ is self-adjoint with respect
to $a_h(\cdot,\cdot)$ and $b_h(\cdot,\cdot)$. To describe the spectrum
of this operator, we proceed as in the continuous case and decompose
$\bV_h$ into a convenient direct sum. To this end, we define
$$
\bK_h:=\bV_h\cap\bK
=\left\{\bv_h\in\bV_h:\ \div\bv_h=0\text{ in }\O\right\}
$$
and notice that, here again,
$\bT_h|_{\bK_h}:\;\bK_h\longrightarrow\bK_h$ reduces to the identity.
Moreover, we have the following result.

\begin{proposition}
$\mu_h=1$ is an eigenvalue of $\bT_h$ and its eigenspace is $\bK_h$.
\end{proposition}

\begin{proof}
We have that $\bw_h\in\bV_h$ is an eigenfunction associated with the
eigenvalue $\mu_h=1$ if and only if 
$\int_{E}\div\bw_h\div\bv_h=0\ \,\forall\bv_h\in\bV_h$, namely, if and
only if $\bw_h\in\bK_h$. 
$\hfill\qed$
\end{proof}

As a consequence of all this, we have the following spectral characterization of the discrete solution
operator.

\begin{theorem}
\label{CHAR_SP_DISC}
The spectrum of $\bT_h$ consists of $M_h:=\dim(\bV_h)$ eigenvalues,
repeated according to their respective multiplicities. It decomposes as
$\sp(\bT_h)=\left\{1\right\}\cup\left\{\mu_{hk}\right\}_{k=1}^{N_h}$,
where:
\begin{enumerate}
\item[\textit{i)}] the eigenspace associated with $\mu_h=1$ is $\bK_h$;
\item[\textit{ii)}] $\mu_{hk}\in(0,1)$, 
$k=1,\dots,N_h:=M_h-\dim(\bK_h)$, are non-defective eigenvalues repeated
according to their respective multiplicities.
\end{enumerate}
\end{theorem}

In what follows, we derive several auxiliary results which will be used
in the following section to prove convergence and error estimates for
the spectral approximation.

First, we establish interpolation properties in the discrete space
$\bV_h$. Although the $\bV_h$-interpolant can be defined for less
regular functions, in our case it is enough to consider $\bv\in\bV$ such
that $\bv|_{E}\in [\HtE]^2$ for some $t>1/2$ and for all $E\in\CT_h$, so
that we can easily take its trace on each individual edge. Then, we
define its interpolant $\bv_I\in\bV_h$ by fixing its degrees of freedom
as follows:
\begin{align}
\label{uno}
\int_{e}\left(\bv-\bv_{I}\right)\cdot\bn\;q\ds & =0
\qquad\forall q\in\bbP_k(e),
\quad\forall\text{ internal edge }e\not\subset\G;
\\
\label{dos}
\int_{E}\left(\bv-\bv_{I}\right)\cdot\nabla q & =0
\qquad\forall q\in\bbP_k(E)/\bbP_{0}(E),
\quad\forall E\in\T_{h}.
\end{align}

In what follows, we state two results about the approximation properties
of this interpolant, whose proof we postpone to the Appendix. The first
one concerns approximation properties of $\div\bv_I$ and follows from a
commuting diagram property for this interpolant, which involves the
$\LO$-orthogonal projection 
$$
P_k:\;\LO\longrightarrow
\left\{q\in\LO: q|_E\in\bbP_k(E)\quad\forall E\in\CT_h\right\}.
$$

\begin{lemma}
\label{lemmainter}
Let $\bv\in\bV$ be such that $\bv\in[\HtO]^2$ with $t>1/2 $. Let
$\bv_I\in\bV_h$ be its interpolant defined by \eqref{uno}--\eqref{dos}.
Then,
$$
\div\bv_I=P_k(\div\bv)\quad\text{ in }\O.
$$
Consequently, for all $E\in\CT_h$,
$\left\|\div\bv_I\right\|_{0,E}\leq\left\|\div\bv\right\|_{0,E}$ and,
if $\div\bv|_{E}\in\HrE$ with $r\geq 0$, then
\begin{equation*}
\left\|\div\bv-\div\bv_I\right\|_{0,E}
\leq Ch_E^{\min\{r,k+1\}}\left|\div\bv\right|_{r,E}. 
\end{equation*}
\end{lemma}

The second result concerns the $\LO$ approximation property of $\bv_I$.

\begin{lemma}
\label{lemmainterV_I}
Let $\bv\in\bV$ be such that $\bv\in[\HtO]^2$ with $t>1/2$. Let
$\bv_I\in\bV_h$ be its interpolant defined by \eqref{uno}--\eqref{dos}.
Let $E\in\CT_h$. If $1\leq t\leq k+1$, then
\begin{equation*}
\left\|\bv-\bv_I\right\|_{0,E}
\leq Ch_E^{t}\left|\bv\right|_{t,E},
\end{equation*}
whereas, if $1/2 <t\leq1$, then
\begin{equation*}
\left\|\bv-\bv_I\right\|_{0,E}
\leq C\left(h_E^{t}\left|\bv\right|_{t,E}
+h_E\left\|\div\bv\right\|_{0,E}\right).
\end{equation*}
\end{lemma}

Let $\bK_h^{\bot}$ be the $[\LO]^2$-orthogonal complement of $\bK_h$ in
$\bV_h$, namely,
$$
\bK_h^{\bot}:=\left\{\bv_h\in\bV_h: 
\ \int_{\O}\bv_h\cdot\boldsymbol{\xi}_h=0
\quad\forall\boldsymbol{\xi}_h\in\bK_h\right\}.
$$
Note that $\bK_h$ and $\bK_h^{\bot}$ are also orthogonal in $\HdivO$.
The following lemma shows that, although
$\bK_h^{\bot}\not\subset\bK^{\bot}=\bG\cap\bV$, the gradient part in the
Helmholtz decomposition of a function in $\bK_h^{\bot}$ is
asymptotically small.

\begin{lemma}
\label{conver}
Let $\bv_h\in\bK_h^{\bot}$. Then, there exist $p\in\HusO$ with
$s\in(1/2,1]$ as in Lemma~\ref{ortho} and $\bpsi\in\bK$ such that
$\bv_h=\bpsi+\nabla p$ and
\begin{align}
\label{i}
\left\|\nabla p\right\|_{s,\O}
& \leq C\left\|\div\bv_h\right\|_{0,\O},
\\
\label{ii}
\left\|\bpsi\right\|_{0,\O}
& \leq Ch^{s}\left\|\div\bv_h\right\|_{0,\O}.
\end{align}
\end{lemma}

\begin{proof}
Let $\bv_h\in\bK_h^{\bot}\subset\bV_h\subset\bV$. As a consequence of
Lemma~\ref{ortho}, we know that there exist $p\in\HusO$ and
$\bpsi\in\bK$ such that $\bv_h=\nabla p+\bpsi$ and that
$\left\|\nabla p\right\|_{s,\O}\leq C\left\|\div\bv_h\right\|_{0,\O}$,
which proves \eqref{i}. 

On the other hand, we have that
$$
\left\|\bpsi\right\|_{0,\O}^2 
=\int_{\O}\left(\nabla p-\bv_h\right)
\cdot\left(\nabla p-(\nabla p)_{I}\right) 
+\int_{\O}\left(\nabla p-\bv_h\right)
\cdot\left( (\nabla p)_I-\bv_h\right).
$$
Now, according to Lemma~\ref{lemmainter}, 
$\div((\nabla p)_{I})=P_k(\div(\nabla p))$. Therefore, since $\Delta
p=\div\bv_h$, we obtain
$$
\div\left((\nabla p)_{I}-\bv_h\right)
=P_k\left(\Delta p\right)-\div\bv_h
=P_k\left(\div\bv_h\right)-\div\bv_h=0,
$$
where we have used that for $\bv_h\in\bV_h$,
$\div\bv_h|_{E}\in\bbP_k(E)$. Therefore 
$\left((\nabla p)_{I}-\bv_h\right)\in\bK_h\subseteq\bK$ and since
$\nabla p\in\bG\cap\bV=\bK^{\bot}$ and $\bv_h\in\bK_h^{\bot}$, we have
that
\begin{equation*}
\int_{\O}\left(\nabla p-\bv_h\right)
\cdot\left((\nabla p)_I-\bv_h\right)=0.
\end{equation*}
Thus,
\begin{equation*}
\left\|\bpsi\right\|_{0,\O}^2
=\int_{\O}\left(\nabla p-\bv_h\right)
\cdot\left(\nabla p-(\nabla p)_{I}\right)
\end{equation*}
and, by using Cauchy-Schwarz inequality, Lemma~\ref{lemmainterV_I} and
\eqref{i}, we obtain
\begin{align*}
\left\|\bpsi\right\|_{0,\O}^2
& \leq\sum_{E\in\T_h}\left\|\nabla p-\bv_h\right\|_{0,E}
\left\|\nabla p-(\nabla p)_{I}\right\|_{0,E}
\\
& \leq C\sum_{E\in\T_h}\left\|\nabla p-\bv_h\right\|_{0,E}
\left(h^s_E\left\|\nabla p\right\|_{s,E}
+h_E\left\|\div(\nabla p)\right\|_{0,E}\right)
\\
& \leq Ch^s\left\|\bpsi\right\|_{0,\O}\left\|\div\bv_h\right\|_{0,\O},
\end{align*}
which allows us to complete the proof.
$\hfill\qed$
\end{proof}

To end this section, we prove the following result which will be used
in the sequel. Let $\bpi_h$ be defined in $\bV$ by 
\begin{equation}
\label{proyectorg}
\left(\bpi_h\bv\right)|_E
:=\bpi_h^E(\bv|_E)
\quad\text{ for all }E\in\T_h
\end{equation}
with $\bpi_h^E$ defined by \eqref{numero}.

\begin{lemma}
\label{A}
There exists a constant $C>0$ such that, for every $p\in\HutO$ with
$1/2<t\le k+1$, there holds 
$$
\left\|\nabla p-\bpi_h (\nabla p)\right\|_{0,\O}
\leq Ch^{t}\left\|\nabla p\right\|_{t,\O}.
$$
\end{lemma}

\begin{proof}
The result follows from the fact that, since $\bpi_h^{E}$ is the
$[\L^2(E)]^2$-projection onto $\widehat{\bV}_h^E:=\nabla(\bbP_{k+1}(E))$
(cf. \eqref{numero}), 
$$
\big\|\nabla p-\bpi_h^E(\nabla p)\big\|_{0,E}
=\inf_{ q\in\bbP_{k+1}(E)}\left\|\nabla p-\nabla q\right\|_{0,E}
\leq Ch_E^{t}\left\|\nabla p\right\|_{t,E}.
$$
Let us remark that the last inequality is a consequence of standard
approximation estimates for polynomials on polygons in case of integer
$t$ (see, for instance, \cite[Lemma~4.3.8]{BS-2008}) and standard Banach
space interpolation results for non-integer $t$. 
$\hfill\qed$
\end{proof}

\section{Spectral approximation and error estimates}
\label{SEC:approximation}

To prove that $\bT_h$ provides a correct spectral approximation of 
$\bT$, we will resort to the theory developed in \cite{Raviart1} for
non-compact operators. To this end, we first introduce some notation.
For any linear bounded operator $\bS:\;\bV\longrightarrow\bV$,
we define 
\begin{equation*}
\left\|\bS\right\|_h
:=\sup_{\0\ne\bv_h\in\bV_h}
\dfrac{\left\|\bS\bv_h\right\|_{\divO}}
{\left\|\bv_h\right\|_{\divO}}.
\end{equation*}

We recall the definition of the \textit{gap} $\hdel$ between two closed
subspaces $\bX$ and $\bY$ of $\bV$:
$$
\hdel(\bX,\bY)
:=\max\left\{\delta(\bX,\bY),\delta(\bY,\bX)\right\},
$$
where
$$
\delta(\bX,\bY)
:=\sup_{\scriptsize\begin{matrix}\bx\in\bX 
\\ \left\|\bx\right\|_{\divO}=1 \end{matrix}}\delta(\bx,\bY)
\qquad\text{with } \delta(\bx,\bY)
:=\inf_{\by\in\bY}\left\|\bx-\by\right\|_{\divO}.
$$
The theory from \cite{Raviart1} guarantees approximation of the spectrum
of $\bT$, provided the following two properties are satisfied:
\begin{itemize}
\item \textbf{P1}: $\ \left\|\bT-\bT_h\right\|_h\rightarrow0\ $
as $\ h\rightarrow0$;
\\[-.2cm]
\item \textbf{P2}: $\ \forall\bv\in\bV\ $
$\ \displaystyle\lim_{h\rightarrow 0}\delta(\bv,\bV_h)=0$.
\end{itemize}

Property \textbf{P2} follows immediately from the density of the smooth
functions in $\bV$ and the approximation properties in
Lemmas~\ref{lemmainter} and \ref{lemmainterV_I}. Hence, there only
remains to prove property \textbf{P1}. With this aim, first we note that
since $\bT|_{\bK_h}$ and $\bT_h|_{\bK_h}$ both reduce to the identity,
it is enough to estimate
$\left\|\left(\bT-\bT_h\right)\bf_h\right\|_{\divO}$ for
$\bf_h\in\bK_h^\bot$.

\begin{lemma}
\label{lemcotste}
There exists $C>0$ such that, for all $\bf_h\in\bK_h^{\bot}$,
\begin{equation*}
\left\|\left(\bT-\bT_h\right)\bf_h\right\|_{\divO}
\leq Ch^s\left\|\bf_h\right\|_{\divO}
\end{equation*}
with $s\in(1/2,1]$ as in Lemma~\ref{ortho}.
\end{lemma}

\begin{proof}
Let $\bf_h\in\bK_h^\bot $, $\bu:=\bT\bf_h$ and $\bu_h:=\bT_h\bf_h$.
According to Lemma~\ref{ortho}, we write $\bu=\bphi+\nabla q$ with
$\bphi\in\bK$, $\nabla q\in [\HsO]^2$ and 
$\left\|\nabla q\right\|_{s,\O}\leq C\left\|\div\bu\right\|_{0,\O}$. We
have 
\begin{equation}
\label{estimate}
\left\|\left(\bT-\bT_h\right)\bf_h\right\|_{\divO}\leq
\left\|\bu-(\nabla q)_I\right\|_{\divO}
+\left\|\bu_h-(\nabla q)_I\right\|_{\divO},
\end{equation}
where $(\nabla q)_I$ is the $\bV_h$-interpolant of $\nabla q$
defined by \eqref{uno}--\eqref{dos}.
We define $\bv_h:=\bu_h-(\nabla q)_I\in\bV_h$. Thanks to
Lemma~\ref{ha-elipt-disc}, the definition \eqref{21} of
$b_h^{E}(\cdot,\cdot)$ and those of $\bT$ and $\bT_h$, we have 
\begin{align*}
\beta\left\|\bv_h\right\|^2_{\divO}
& \leq a_h(\bv_h,\bv_h)
=a_h(\bu_h,\bv_h)-a_h((\nabla q)_I,\bv_h)
\\
& =b_h(\bf_h,\bv_h)-\int_{\O}\div((\nabla q)_I)\div\bv_h
-\sum_{E\in\T_h}b_h^E((\nabla q)_I,\bv_h) 
\\
& =b_h(\bf_h,\bv_h)-\int_{\O}\bf_h\cdot\bv_h
+\int_{\O}\div(\bu-(\nabla q)_I)\div\bv_h
\\
& \hphantom{=}
-\sum_{E\in\T_h}\left(b_h^E\left((\nabla q)_I-\bpi_h^E\bu,\bv_h\right)
+\int_{E}\left(\bpi_h^E\bu-\bu\right)\cdot\bv_h\right),
\end{align*}
where for the last equality we have also used the consistency property
\eqref{consistencia1}. Since $\div((\nabla q)_I)=P_k(\div(\nabla q))$
(cf. Lemma~\ref{lemmainter}), we have that $\int_{\O}\div(\bu-(\nabla
q)_I)\div\bv_h=0$ for all $\bv_h\in\bV_h$. Then,
\begin{multline}
\beta\left\|\bv_h\right\|^2_{\divO}
\le\left(b_h(\bf_h,\bv_h)-\int_{\O}\bf_h\cdot\bv_h\right)
\\
\label{cotaa}
-\sum_{E\in\T_h}
\left(b_h^E((\nabla q)_I-\bpi_h^E\bu,\bv_h)
+\int_{E}\left(\bpi_h^E\bu-\bu\right)\cdot\bv_h\right).
\end{multline}
The first term on the right hand side can be bounded as follows: 
\begin{align*}
\label{51}
b_h(\bf_h,\bv_h) 
& -\int_{\O}\bf_h\cdot\bv_h
=\sum_{E\in\CT_h}
\left(b_h^E(\bf_h,\bv_h)-\int_{E}\bf_h\cdot\bv_h\right)
\\
\nonumber
& =\sum_{E\in\CT_h}
\left(\int_{E}\bpi_h^E\bf_h\cdot\bpi_h^E\bv_h
+S^E\big(\bf_h-\bpi_h^E\bf_h,\bv_h-\bpi_h^E\bv_h\big)
-\int_{E}\bf_h\cdot\bv_h\right)
\\
& =\sum_{E\in\CT_h}
\int_{E}\left(\bpi_h^E\bf_h-\bf_h\right)\cdot\bv_h
+\sum_{E\in\CT_h}S^E\big(\bf_h-\bpi_h^E\bf_h,\bv_h-\bpi_h^E\bv_h\big),
\nonumber
\end{align*}
where we have used \eqref{numero} to write the last equality. Now, from
the symmetry of $S^E(\cdot,\cdot)$, \eqref{20}, a Cauchy-Schwarz
inequality and the fact that $\bpi_h^E$ is an $\L^2(E)$-projection (cf.
\eqref{numero}), we have that
\begin{equation*}
\label{55}
\sum_{E\in\CT_h}S^E\big(\bf_h-\bpi_h^E\bf_h,\bv_h-\bpi_h^E\bv_h\big)
\leq\sum_{E\in\CT_h}
c_1\big\|\bf_h-\bpi_h^E\bf_h\big\|_{0,E}
\left\|\bv_h\right\|_{0,E}.
\end{equation*}
Therefore, using Cauchy-Schwarz inequality again,
\begin{equation}
\label{estabilitiss}
b_h(\bf_h,\bv_h)-\int_{\O}\bf_h\cdot\bv_h\leq C\sum_{E\in\CT_h}
\big\|\bf_h-\bpi_h^E\bf_h\big\|_{0,E}
\left\|\bv_h\right\|_{0,E}.
\end{equation}
Substituting the above estimate in \eqref{cotaa}, from
\eqref{consistencia2} and Cauchy-Schwarz inequality we obtain
\begin{align*}
\beta\left\|\bv_h\right\|^2_{\divO}
& \leq C\sum_{E\in\CT_h}
\left(\big\|\bf_h-\bpi_h^E\bf_h\big\|_{0,E}
+\left\|\bu-(\nabla q)_I\right\|_{0,E}
+\big\|\bu-\bpi_h^E\bu\big\|_{0,E}\right)
\left\|\bv_h\right\|_{0,E}
\\
& \leq C\left(\left\|\bf_h-\bpi_h\bf_h\right\|_{0,\O}
+\left\|\bu-(\nabla q)_I\right\|_{0,\O}
+\left\|\bu-\bpi_h\bu\right\|_{0,\O}\right)
\left\|\bv_h\right\|_{\divO},
\end{align*}
with $\bpi_{h}$ as defined in \eqref{proyectorg}. Therefore, from
\eqref{estimate},
$$
\left\|\left(\bT-\bT_h\right)\bf_h\right\|_{\divO}
\leq C\left(\left\|\bf_h-\bpi_h\bf_h\right\|_{0,\O}
+\left\|\bu-\bpi_h\bu\right\|_{0,\O}
+\left\|\bu-(\nabla q)_I\right\|_{\divO}\right).
$$

Thus, there only remains to estimate the three terms on the right-hand
side above. For the first one we write $\bf_h=\bpsi+\nabla p$ with
$\psi\in\bK$ and $p\in\HusO$ as in Lemma~\ref{conver}. Hence, by using
this and Lemma~\ref{A},
\begin{align*}
\left\|\bf_h-\bpi_h\bf_h\right\|_{0,\O} 
& \leq\left\|\bpsi-\bpi_h\bpsi\right\|_{0,\O}
+\left\|\nabla p-\bpi_h (\nabla p)\right\|_{0,\O}
\\
& \leq C\left(\left\|\bpsi\right\|_{0,\O}
+\left\|\nabla p-\bpi_h(\nabla p)\right\|_{0,\O}\right)
\\
& \leq Ch^s\left\|\div\bf_h\right\|_{0,\O}.
\end{align*}

On the other hand, we have that 
$\bu=\bT(\bpsi+\nabla p)=\bpsi+\bT(\nabla p)$ and, from
Lemmas~\ref{invariant} and \ref{ortho}, $\bT(\nabla p)=\nabla q$ and
$\bpsi=\bphi$. Moreover, by virtue of Theorem~\ref{teo2.5}, $q\in\HusO$
and
$$
\left\|\nabla q\right\|_{s,\O}
\leq C\left\|\nabla p\right\|_{\divO}
\leq C\left\|\bf_h\right\|_{\divO},
$$
whereas estimate \eqref{ii} still holds true for $\bpsi$:
$$
\left\|\bpsi\right\|_{0,\O}
\leq Ch^s\left\|\div\bf_h\right\|_{0,\O}.
$$ 
Then, using that $\bpi_h$ is an $[\LO]^2$-projection, from
Lemmas~\ref{conver} and \ref{A} we have
\begin{align*}
\left\|\bu-\bpi_h\bu\right\|_{0,\O} 
& \leq\left\|\bpsi-\bpi_h\bpsi\right\|_{0,\O}
+\left\|\nabla q-\bpi_h (\nabla q)\right\|_{0,\O})
\\
& \leq C\left(\left\|\bpsi\right\|_{0,\O}
+\left\|\nabla q-\bpi_h(\nabla q)\right\|_{0,\O}\right)
\\
& \leq Ch^s\left\|\div\bf_h\right\|_{0,\O}
+Ch^s\left\|\nabla q\right\|_{s,\O}
\\
& \leq Ch^s\left\|\bf_h\right\|_{\divO}.
\end{align*} 
Finally, using once more that $\bu=\bpsi+\nabla q$ and
Lemmas~\ref{conver}, \ref{lemmainterV_I} and \ref{lemmainter}, we write
\begin{align*}
& \left\|\bu-(\nabla q)_I\right\|_{\divO}
\leq\left\|\bpsi\right\|_{\divO}
+\left\|\nabla q-(\nabla q)_I\right\|_{\divO}
\\
& \hphantom{\left\|\bu-(\nabla q)_I\right\|}
\leq Ch^s\left\|\div\bf_h\right\|_{0,\O}
+\left\|\nabla q-(\nabla q)_I\right\|_{0,\O}
+\left\|\div(\nabla q)-\div((\nabla q)_I)\right\|_{0,\O}
\\
& \hphantom{\left\|\bu-(\nabla q)_I\right\|}
\leq Ch^s\left\|\bf_h\right\|_{\divO}
+C\left(h^s\left|\nabla q\right|_{s,\O}
+h\left\|\div(\nabla q)\right\|_{0,\O}\right)
+Ch\left|\div(\nabla q)\right|_{1,\O}
\\
& \hphantom{\left\|\bu-(\nabla q)_I\right\|}
\leq Ch^s\left\|\bf_h\right\|_{\divO},
\end{align*}
where, we have used that $\nabla q=\bT(\nabla p)$ and, hence, since
$\nabla p\in\bG\cap\bV$, from Theorem~\ref{teo2.5}
$\div(\nabla q)\in\HuO$ and $\left\|\div(\nabla q)\right\|_{1,\O} 
\le C\left\|\nabla p\right\|_{\divO}\le C\left\|\bf_h\right\|_{\divO}$. 

Collecting the previous estimates, we obtain
\begin{align*}
\left\|\left(\bT-\bT_h\right)\bf_h\right\|_{\divO}
& \leq Ch^s\left\|\bf_h\right\|_{\divO}
\end{align*}
and we end the proof.
$\hfill\qed$
\end{proof}

Now, we are in a position to conclude property \textbf{P1}.

\begin{corollary}
\label{cotaT}
There exists $C>0$, independent of $h$, such that
\begin{equation*}
\left\|\bT-\bT_h\right\|_{h}\leq Ch^s.
\end{equation*}
\end{corollary}

\begin{proof}
Given $\bv_h\in\bV_h$, we have that $\bv_h=\bpsi_h+\bf_h$ with
$\bpsi_h\in\bK_h$ and $\bf_h\in\bK_h^{\bot}$, then
\begin{equation*}
\left\|\left(\bT-\bT_h\right)\bv_h\right\|_{\divO}
=\left\|\left(\bT-\bT_h\right)\bf_h\right\|_{\divO}
\leq Ch^s\left\|\bf_h\right\|_{\divO},
\end{equation*}
where the last inequality follows from Lemma~\ref{lemcotste}. The proof
follows by noting that, since $\bv_h=\bpsi_h+\bf_h$ is an orthogonal
decomposition in $\HdivO$, we have that
$\left\|\bf_h\right\|_{\divO}\le\left\|\bv_h\right\|_{\divO}$.
$\hfill\qed$
\end{proof}

In order to establish spectral convergence and error estimates, we 
recall some other basic definitions from spectral theory.

Given a generic linear bounded operator $\bS:\;\bV\longrightarrow\bV$
defined on a Hilbert space $\bV$, the spectrum of $\bS$ is the set
$\sp(\bS):=\left\{z\in\C:
\ \left(z\bI-\bS\right)\text{ is not invertible}\right\}$ and the
resolvent set of $\bS$ is its complement
$\rho(\bS):=\C\setminus\sp(\bS)$. For any $z\in\rho(\bS)$,
$R_z(\bS):=\left(z\bI-\bS\right)^{-1}:\;\bV\longrightarrow\bV$ is the
resolvent operator of $\bS$ corresponding to $z$.

The following two results are consequence of property \textbf{P1}, see
\cite[Lemma~1 and Theorem~1]{Raviart1}.

\begin{lemma}
\label{lema1RT} Let us assume that {\rm\textbf{P1}} holds true and let
$F\subset\rho(\bT) $ be closed. Then, there exist positive constants $C$
and $h_0$ independent of $h$, such that for $h<h_0$ 
$$
\sup_{\bv_h\in\bV_h}\left\|R_z(\bT_{h})\bv_{h}\right\|_{\divO}
\leq C\left\|\bv_{h}\right\|_{\divO}
\qquad\forall z\in F.
$$
\end{lemma}

\begin{theorem}
Let $U\subset\C$ be an open set containing $\sp(\bT)$. Then, there
exists $h_0>0$ such that $\sp(\bT_h)\subset U$ for all $h<h_0$.
\end{theorem}

An immediate consequence of this theorem and Corollary~\ref{cotaT} is
that the proposed virtual element method does not introduce spurious
modes with eigenvalues interspersed among those with a physical meaning.
Let us remark that such a spectral pollution could be in principle
expected from the fact that the corresponding solution operator $\bT$
has an infinite-dimensional eigenvalue $\mu=1$ (see
\cite{BDMR95,BHPR2001,Boffi}).

By applying the results from \cite[Section 2]{Raviart1} to our problem,
we conclude the spectral convergence of $\bT_h$ to $\bT$ as $h\to0$.
More precisely, let $\mu\in(0,1)$ be an isolated eigenvalue of $\bT$
with multiplicity $m$ and let $\CC$ be an open circle in the complex
plane centered at $\mu$, such that $\mu$ is the only eigenvalue of $\bT$
lying in $\CC$ and $\partial\CC\cap\sp(\bT)=\emptyset$. Then, according
to \cite[Section 2]{Raviart1}, for $h$ small enough there exist $m$
eigenvalues $\mu^{(1)}_h,\dots,\mu^{(m)}_h$ of $\bT_h$ (repeated
according to their respective multiplicities) which lie in $\CC$. Therefore, these eigenvalues $\mu^{(1)}_h,\dots,\mu^{(m)}_h$ converge to $\mu$ as $h$ goes to zero.

Our next step is to obtain error estimates for the spectral
approximation. The classical reference for this issue on non-compact
operators is \cite{Raviart2}. However, we cannot apply the results from
this reference directly to our problem, because of the variational
crimes in the bilinear forms used to define the operator $\bT_h$.
Therefore, we need to extend the results from this reference to our
case. With this purpose, we follow an approach inspired by those of
\cite{BDRS,LMR}.

Consider the eigenspace $\bE$ of $\bT$ corresponding to $\mu$ and the
$\bT_h$-invariant subspace $\bE_h$ spanned by the eigenspaces of $\bT_h$
corresponding to $\mu^{(1)}_h,\dots,\mu^{(m)}_h$. As a consequence of
Lemma~\ref{lema1RT}, we have for $h$ small enough
\begin{equation}
\label{resol}
\left\|\left(z\bI-\bT_h\right)\bv_h\right\|_{\divO}
\geq C\left\|\bv_h\right\|_{\divO}
\qquad\forall\bv_h\in\bV_h,
\quad\forall z\in\partial\CC.
\end{equation}
Let $\bP_h:\;\bV\longrightarrow\bV_{h}\hookrightarrow\bV$ be the projector
with range $\bV_h$ defined by the relation 
$$
a(\bP_h\bu-\bu,\bv_h)=0
\qquad\forall\bv_h\in\bV_h.
$$
In our case, the bilinear form $a(\cdot,\cdot)$ is the inner product of
$\bV$, so that
$\left\|\bP_h\bu\right\|_{\divO}\leq\left\|\bu\right\|_{\divO}$ and 
$$
\left\|\bu-\bP_h\bu\right\|_{\divO}
=\delta(\bu,\bV_{h})\qquad\forall\bu\in\bV.
$$

Now, we define
$\widehat{\bT}_{h}:=\bT_{h}\bP_h:\;\bV\longrightarrow\bV_{h}$. Notice
that $\sp(\widehat{\bT}_h)=\sp(\bT_h)\cup\{0\}$. Furthermore, we have
the following result (cf. \cite[Lemma~1]{Raviart2}).

\begin{lemma}
\label{cotare} 
There exist $h_0>0$ and $C>0$ such that
$$
\big\|R_z(\widehat{\bT}_h)\big\|_{\divO}\leq C
\qquad\forall z\in\partial\CC,
\quad\forall h\leq h_{0}.
$$
\end{lemma}

\begin{proof}
Since $\widehat{\bT}_h$ is compact, it suffices to check that
$\big\|(z\bI-\widehat{\bT}_h)\bv\big\|_{\divO}
\geq C\left\|\bv\right\|_{\divO}$ $\,\forall\bv\in\bV$ and 
$\,\forall z\in\partial\CC$. By using \eqref{resol} and basic properties
of the projector $\bP_h$, we obtain
\begin{align*}
\left\|\bv\right\|_{\divO} 
& \leq\left\|\bP_h\bv\right\|_{\divO}
+\left\|\bv-\bP_h\bv\right\|_{\divO}
\\
& \leq C\left\|\left(z\bI-\bT_{h}\right)\bP_h\bv\right\|_{\divO}
+\left|z\right|^{-1}\left\|z\left(\bv-\bP_h\bv\right)\right\|_{\divO}
\\
& \leq C\big\|\big(z\bI-\widehat{\bT}_h\big)\bP_h\bv\big\|_{\divO}
+\left|z\right|^{-1}\big\|z\left(\bv-\bP_h\bv\right)-\widehat{\bT}_{h}
\left(\bv-\bP_h\bv\right)\big\|_{\divO}
\\
& =C\big\|\bP_h\big(z\bI-\widehat{\bT}_h\big)\bv\big\|_{\divO}
+\left|z\right|^{-1}\big\|\left(\bI-\bP_h\right)
\big(z\bI-\widehat{\bT}_h\big)\bv)\big\|_{\divO}
\\
& \leq C\big\|\big(z\bI-\widehat{\bT}_h\big)\bv\big\|_{\divO},
\end{align*}
where we have used that the curve $\partial\CC$ is bounded away from
$0$.
$\hfill\qed$
\end{proof}

Next, we introduce the following spectral projectors (the second one,
is well defined at least for $h$ small enough):
\begin{itemize}
\item the spectral projector of $\bT$ relative to $\mu$:\quad
$\displaystyle\bF:=\dfrac{1}{2\pi i}\int_{\partial\CC}R_z(\bT)\,dz$;
\item the spectral projector of $\widehat{\bT}_h$
relative to $\mu_{h}^{(1)},\ldots,\mu_{h}^{(m)}$:
\quad $\widehat{\bF}_h:=\dfrac{1}{2\pi i}
\displaystyle\int_{\partial\CC}R_z(\widehat{\bT}_h)\,dz$.
\end{itemize}
We also introduce the quantities
$$
\gamma_{h}:=\delta(\bE,\bV_{h})
\qquad\text{ and }\qquad
\eta_h:=\sup_{\bw\in\bE}
\dfrac{\left\|\bw-\bpi_h\bw\right\|_{0,\O}}
{\left\|\bw\right\|_{\divO}}.
$$
These two quantities are bounded as follows:
\begin{equation}
\label{cotafin}
\gamma_{h}\leq Ch^{\min\{\tilde{s},k+1\}}
\qquad\text{ and }\qquad
\eta_h\leq Ch^{\min\{\tilde{s},k+1\}},
\end{equation}
where $\tilde{s}>1/2$ is such that $\bE\subset[\H^{\tilde{s}}(\O)]^2$
(cf. Theorem~\ref{CHAR_SP}). In fact, the first estimate follows from
Lemmas~\ref{lemmainter} and \ref{lemmainterV_I} and
Theorem~\ref{CHAR_SP}(ii), whereas the latter follows from the fact that
$\bE\subset\bG\cap\bV$, Lemma~\ref{A} and Theorem~\ref{CHAR_SP}(ii)
again.

The following estimate is a variation of Lemma~3 from \cite{Raviart2})
that will be used to prove convergence of the eigenspaces.

\begin{lemma}
\label{espaciosor}
There exist positive constants $h_0$ and $C$ such that, for all $h<h_0$,
$$
\big\|\big(\bF-\widehat{\bF}_h\big)|_{\bE}\big\|_{\divO}
\leq C\big\|\big(\bT-\widehat{\bT}_h\big)|_{\bE}\big\|_{\divO}
\leq C\left(\gamma_{h}+\eta_h\right).
$$
\end{lemma}

\begin{proof}
The first inequality is proved using the same arguments of
\cite[Lemma~3]{Raviart2} and Lemma~\ref{cotare}. For the other estimate,
let $\bf\in\bE$, $\bw:=\bT\bf$ and
$\bw_h:=\widehat{\bT}_h\bf=\bT_h\bP_h\bf$. Note that, by
Theorem~\ref{ortho}(ii), $\bf\in\nabla(\H^{1+\tilde{s}}(\O))$,
$\tilde{s}>1/2$. By using the first Strang lemma (see, for instance,
\cite[Theorem~4.1.1]{ciarlet}), we have
\begin{align*}
\left\|\bw-\bw_h\right\|_{\divO} 
& \leq C\left(\left\|\bw-\bP_h\bw\right\|_{\divO}
+\sup_{\bv_h\in\bV_h}
\dfrac{\left|b(\bP_h\bw,\bv_h)-b_h(\bP_h\bw,\bv_h)\right|}
{\left\|\bv_h\right\|_{\divO}}\right.
\\
& \quad\quad\quad
+\left.\sup_{\bv_h\in\bV_h}
\dfrac{\left|b(\bf,\bv_h)-b_h(\bP_h\bf,\bv_h)\right|}
{\left\|\bv_h\right\|_{\divO}}\right)
\end{align*}
and by proceeding as in the proof of Lemma~\ref{lemcotste} to derive
\eqref{estabilitiss}, we obtain
\begin{align*}
\left|b(\bP_h\bw,\bv_h)-b_h(\bP_h\bw,\bv_h)\right| 
& \leq C\sum_{E\in\CT_h}
\big\|\bP_h\bw-\bpi_h^E\bP_h\bw\big\|_{0,E}
\left\|\bv_h\right\|_{0,E}
\\
& \leq C\sum_{E\in\CT_h}
\big\|\big(\bI-\bpi_h^E\big)
\left(\bP_h\bw-\bw\right)
+\big(\bI-\bpi_h^E\big)\bw\big\|_{0,E}
\\
& \leq C\left(\left\|\bw-\bP_h\bw\right\|_{0,\O}
+\left\|\bw-\bpi_h\bw\right\|_{0,\O}\right)
\left\|\bv_h\right\|_{\divO}.
\end{align*}
On the other hand,
\begin{align*}
\left|b(\bf,\bv_h)-b_h(\bP_h\bf,\bv_h)\right| 
& \leq\left|b(\bf-\bP_h\bf,\bv_h)\right|
+\left|b(\bP_h\bf,\bv_h)-b_h(\bP_h\bf,\bv_h)\right|
\\
& \leq C\left(\left\|\bf-\bP_h\bf\right\|_{0,\O}
\left\|\bv_h\right\|_{0,\O}\right)
+\left|b(\bP_h\bf,\bv_h)-b_h(\bP_h\bf,\bv_h)\right|
\\
& \leq C\left(\left\|\bf-\bP_h\bf\right\|_{0,\O}
+\left\|\bf-\bpi_h\bf\right\|_{0,\O}\right)\left\|\bv_h\right\|_{\divO},
\end{align*}
where, for the last inequality, we have used the same argument as above.
Then, we have
\begin{align*}
\left\|\bw-\bw_h\right\|_{\divO}
& \leq C\left(\left\|\bw-\bP_h\bw\right\|_{\divO}
+\left\|\bw-\bpi_h\bw\right\|_{0,\O}\right.
\\
& \phantom{\leq C\Big(}
\left.+\left\|\bf-\bP_h\bf\right\|_{0,\O}
+\left\|\bf-\bpi_h\bf\right\|_{0,\O}\right)
\\
& \leq C\left(\gamma_h+\left\|\bw-\bpi_h\bw\right\|_{0,\O}
+\left\|\bf-\bpi_h\bf\right\|_{0,\O}\right)
\\
& =C\big(\gamma_h+\big(1+\mu^{-1}\big)
\left\|\bw-\bpi_h\bw\right\|_{0,\O}\big)
\\
& \leq C\left(\gamma_h+\eta_h\right),
\end{align*}
where we have used that, for $\bf\in\bE$, $\bw:=\bT\bf=\mu\bf$. Thus, we
conclude the proof.
$\hfill\qed$
\end{proof}

To prove an error estimate for the eigenspaces, we also need the
following result.

\begin{lemma}
\label{inversa}
Let 
$$
\boldsymbol{\Lambda}_h
:=\widehat{\bF}_h|_{\bE}:\;\bE\longrightarrow\bE_h.
$$
For $h$ small enough, the operator $\boldsymbol{\Lambda}_h$ is
invertible and there exists $C$ independent of $h$ such that
$$
\big\|\boldsymbol{\Lambda}_h^{-1}\big\|\leq C.
$$
\end{lemma}

\begin{proof}
It follows by proceeding as in the proof of Lemma~2 from
\cite{Raviart2}, by using Lemma~\ref{espaciosor} and the fact that
$\gamma_{h}\rightarrow 0$ and $\eta_{h}\rightarrow 0$ as $h\rightarrow0$
(cf. \eqref{cotafin}).
$\hfill\qed$
\end{proof}

The following theorem shows that the eigenspace of $\bT_h$ (which
coincides with that of $\widehat{\bT}_h$) approximates the eigenspace of
$\bT$.

\begin{theorem}
\label{cotaespaci}
There exists $C>0$ such that, 
\begin{equation*}
\hdel(\bE,\bE_h)\leq C\left(\gamma_h+\eta_h\right).
\end{equation*}
\end{theorem}

\begin{proof}
It follows by arguing exactly as in the proof of Theorem~1 from
\cite{Raviart2} and using Lemmas~\ref{espaciosor} and \ref{inversa}.
$\hfill\qed$
\end{proof}

Finally, we will prove a double-order error estimate for the
eigenvalues. With this aim, let $\l:=\dfrac{1}{\mu}-1$ be the eigenvalue
of Problem~\ref{P1} with eigenspace $\bE$. Let
$\l_{h}^{i}:=\dfrac{1}{\mu_{h}^{i}}-1$, $i=1,\ldots,m$, be the
eigenvalues of Problem~\ref{P3} with invariant subspace $\bE_{h}$. We
have the following result.

\begin{theorem}
There exist positive constants $C$ and $h_0$ independent of $h$, such that, for $h<h_0$,
\begin{equation*}
\big|\l-\l_h^{(i)}\big|
\le C\,\big(\gamma_h^2+\eta_h^2\big),\qquad i=1,\ldots,m.
\end{equation*}
\end{theorem}

\begin{proof}
Let $\bw_h\in\bE_{h}$ be an eigenfunction corresponding to one of the
eigenvalues $\l_{h}^{(i)}$ ($i=1,\dots,m$) with
$\left\|\bw_h\right\|_{\divO}=1$. According to Theorem~\ref{cotaespaci},
$\delta(\bw_{h},\bE)\leq C\left(\gamma_h+\eta_h\right)$. It follows that
there exists $\bw\in\bE$ such that 
\begin{equation}
\label{fin4}
\left\|\bw-\bw_h\right\|_{\divO}
\le C\left(\gamma_h+\eta_h\right).
\end{equation}
Moreover, it is easy to check that $\bw$ can be chosen normalized in
$\HdivO$-norm.

{}From the symmetry of the bilinear forms and the facts that $\bw$ and
$\bw_h$ are solutions of Problem~\ref{P1} and \ref{P3}, respectively, we
have
\begin{multline*}
\int_\O\div(\bw-\bw_h)^2-\l\int_\O\left(\bw-\bw_h\right)^2
=\l_h^{(i)} b_h(\bw_h,\bw_h)-\l\,b(\bw_h,\bw_h)
\\
=\l_h^{(i)}\left(b_h(\bw_h,\bw_h)-b(\bw_h,\bw_h)\right)
+\big(\l_h^{(i)}-\l\big)\,b(\bw_h,\bw_h),
\end{multline*}
from which we obtain the following identity:
\begin{align}
\nonumber
\big(\l_h^{(i)}-\l\big)\,b(\bw_h,\bw_h)
& =\int_\O\div(\bw-\bw_h)^2-\l\int_\O\left(\bw-\bw_h\right)^2
\\
\label{45}
& \quad-\l_h^{(i)}\left(b_{h}(\bw_h,\bw_h)-b(\bw_h,\bw_h)\right).
\end{align}

The next step is to estimate each term on the right hand side above. The
first and the second ones are easily bounded by using the Cauchy-Schwarz
inequality and \eqref{fin4}:
\begin{equation}
\label{terms_12}
\left|\int_\O\div(\bw-\bw_h)^2
-\l\int_\O\left(\bw-\bw_h\right)^2\right|
\le C\left\|\bw-\bw_h\right\|_{\divO}^2
\le C\left(\gamma_h^2+\eta_h^2\right).
\end{equation}
For the third term, we use \eqref{20}--\eqref{21} to write
\begin{align*}
& \left|b_h(\bw_h,\bw_h)-b(\bw_h,\bw_h)\right|
\\
& \qquad
=\left|\sum_{E\in\CT_h}\left(\int_{E}\big(\bpi_h^E\bw_h\big)^{2}
+S^E\big(\bw_h-\bpi_h^E\bw_h,\bw_h-\bpi_h^E\bw_h\big)\right)
-\sum_{E\in\CT_h}\int_{E}\left(\bw_h\right)^{2}\right|
\\
& \qquad\leq\left|\sum_{E\in\CT_h}
\big(\big\|\bpi_h^E\bw_h\big\|_{0,E}^{2}
-\left\|\bw_h\right\|_{0,E}^{2}\big)\right|
+\sum_{E\in\CT_h}c_{1}\int_{E}\big(\bw_h-\bpi_h^E\bw_h\big)^{2}
\\
& \qquad=\sum_{E\in\CT_h}
\big\|\bw_h-\bpi_h^E\bw_h\big\|_{0,E}^{2}
+c_{1}\sum_{E\in\CT_h}\,\big\|\bw_h-\bpi_h^E\bw_h\big\|_{0,E}^{2}
\\
& \qquad\leq C\left\|\bw_h-\bpi_h\bw_h\right\|^2_{0,\O}
\\
& \qquad\leq C\left(\left\|\bw_h-\bw\right\|^2_{0,\O}
+\left\|\bw-\bpi_h\bw\right\|^2_{0,\O}
+\left\|\bpi_h(\bw-\bw_{h})\right\|^2_{0,\O}\right).
\end{align*}
Then, from the last inequality, the definition of $\eta_{h}$, the fact
that $\bpi_h$ is an $[\LO]^2$-projection and \eqref{fin4}, we obtain
\begin{equation}
\label{4.39}
\left|b(\bw_h,\bw_h)-b_h(\bw_h,\bw_h)\right|
\le C\big(\gamma_h^2+\eta_h^2\big).
\end{equation}
On the other hand, from the stability property \eqref{consistencia2},
\begin{align*}
\left\|\div\bw_h\right\|_{0,\O}^2
=\l_h^{(i)}b_h(\bw_h,\bw_h)
\leq\l_h^{(i)}\alpha^*\left\|\bw_h\right\|_{0,\O}^2, 
\end{align*}
hence 
\begin{equation*}
\label{439}
\big(1+\l_h^{(i)}\alpha^*\big)\left\|\bw_h\right\|_{0,\O}^2
\geq\left\|\bw_h\right\|_{\divO}^2=1.
\end{equation*}
Therefore, since $\l_h^{(i)}\to\l$ as $h$ goes to zero, the theorem
follows from \eqref{45}, \eqref{terms_12}, \eqref{4.39} and the
inequality above.
$\hfill\qed$
\end{proof}

As shown in Theorem~\ref{CHAR_SP}(ii), the eigenfunctions satisfy
additional regularity. The following result shows that this implies an
optimal order of convergence for the numerical method.

\begin{corollary}
If $\bE\subset [\H^{\tilde{s}}(\O)]^2$ with $\tilde{s}>1/2$, then
\begin{equation*}
\hdel(\bE,\bE_h)\leq Ch^{\min\{\tilde{s},k+1\}}.
\end{equation*}
and
\begin{equation*}
\big|\l-\l_h^{(i)}\big|
\le Ch^{2\min\{\tilde{s},k+1\}},
\qquad i=1,\ldots,m.
\end{equation*}
\end{corollary}

\begin{proof}
It follows from the above theorems and the estimates \eqref{cotafin}.
$\hfill\qed$
\end{proof}

\section{Numerical results} 
\label{SEC:NUMER} 

Following the ideas proposed in \cite{BBMR2014}, we have implemented in
a MATLAB code a lowest-order VEM ($k=0$) on arbitrary polygonal meshes.
We report in this section a couple of numerical tests which allowed us
to assess the theoretical results proved above. 

To complete the choice of the VEM, we had to fix the bilinear form
$S^{E}(\cdot,\cdot)$ satisfying \eqref{20} to be used. To do this, we
proceeded as in \cite{BBCMMR2013}. For each element $E\in\T_{h}$ with
edges $e_1,\ldots,e_{N_E}$, let
$\left\{\bphi_1,\ldots,\bphi_{N_E}\right\}$ be the dual basis of
$\bV_h^E$ associated with the degrees of freedom \eqref{freedom};
namely, $\bphi_i\in\bV_h^E$ are such that
$$
\int_{e_j}\bphi_i\cdot\bn\ds
=\delta_{ij},
\qquad i,j=1,\ldots,N_E.
$$
Therefore, $\left\|\bphi_{i}\right\|_{\infty,E}\simeq\dfrac1{h_E}$,
namely, there exists $C>0$ such that
\begin{equation*}
\dfrac{1}{Ch_E}
\leq\left\|\bphi_{i}\right\|_{\infty,E}
\leq\dfrac{C}{h_{E}},
\qquad i=1,\ldots,N_E.
\end{equation*}
Hence, a natural choice for $S^{E}(\cdot,\cdot)$ is given by 
\begin{equation*}
\label{estab}
S^{E}(\bu_h,\bv_h)
:=\sigma_E\sum_{k=1}^{N_E}
\left(\int_{e_k}\bu_h\cdot\bn\right)
\left(\int_{e_k}\bv_h\cdot\bn\right),
\qquad\bu_h,\bv_h\in\bV^{E}_h,
\end{equation*}
where $\sigma_E$ is the so-called \textit{stability constant} which will
be taken of the order of unity (see for instance \cite{BBCMMR2013}).

\subsection{Test 1: Rectangular acoustic cavity}

In this test, the domain is a rectangle $\O:=(0,a)\times(0,b)$, in which
case the exact analytic solution is known. The non vanishing
eigenvalues of Problem~\ref{P1} are given by
\begin{equation*}
\l_{nm}:=\pi^{2}\left(\left(\dfrac{n}{a}\right)^{2}
+\left(\dfrac{m}{b}\right)^{2}\right),
\qquad n,m=0,1,2,\ldots,\quad n+m\neq 0,
\end{equation*}
while the corresponding eigenfunctions are
$$
\bw_{nm}:=\begin{pmatrix}
\dfrac{n}{a}\sin\dfrac{n\pi x}{a}\cos\dfrac{m\pi y}{b}
\\[0.35cm]
\dfrac{m}{b}\cos\dfrac{n\pi x}{a}\sin\dfrac{m\pi y}{b}
\end{pmatrix}.
$$
We have used $a=1$ and $b=1.1$. The stability constant has been taken
$\sigma_E=1$. We have used three different families of meshes (see
Figure~\ref{FIG:VM35}):
\begin{itemize}
\item $\CT_h^1$: triangular meshes;
\\[-.3cm]
\item $\CT_h^2$: rectangular meshes;
\\[-.3cm]
\item $\CT_h^3$: hexagonal meshes.
\end{itemize}
The refinement parameter $N$ used to label each mesh is the number of
elements intersecting each edge. 

\begin{figure}[H]
\begin{center}
\begin{minipage}{4.2cm}
\centering\includegraphics[height=4.1cm, width=4.1cm]{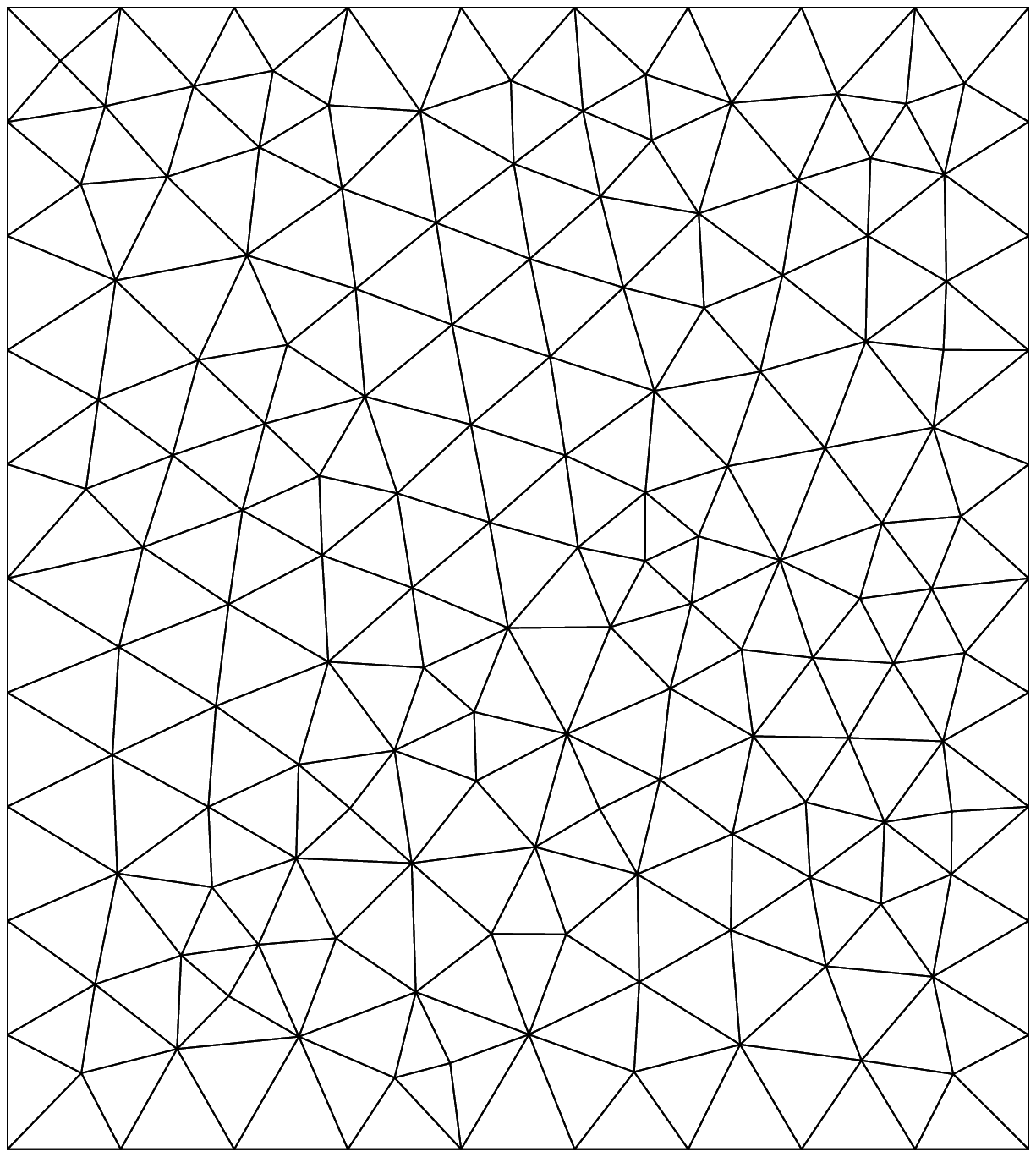}
\end{minipage}
\begin{minipage}{4.2cm}
\centering\includegraphics[height=4.1cm, width=4.1cm]{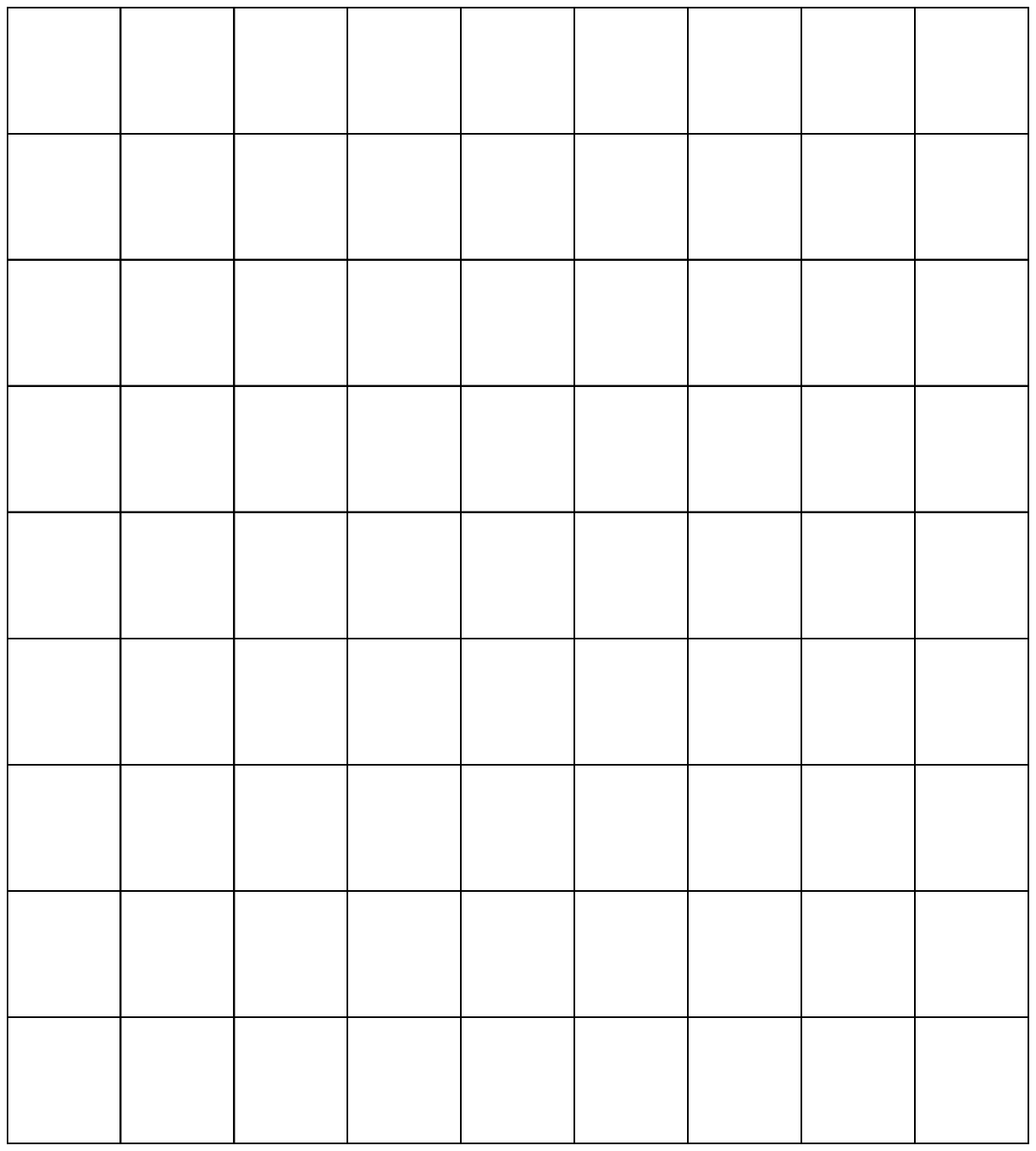}
\end{minipage}
\begin{minipage}{4.2cm}
\centering\includegraphics[height=4.1cm, width=4.1cm]{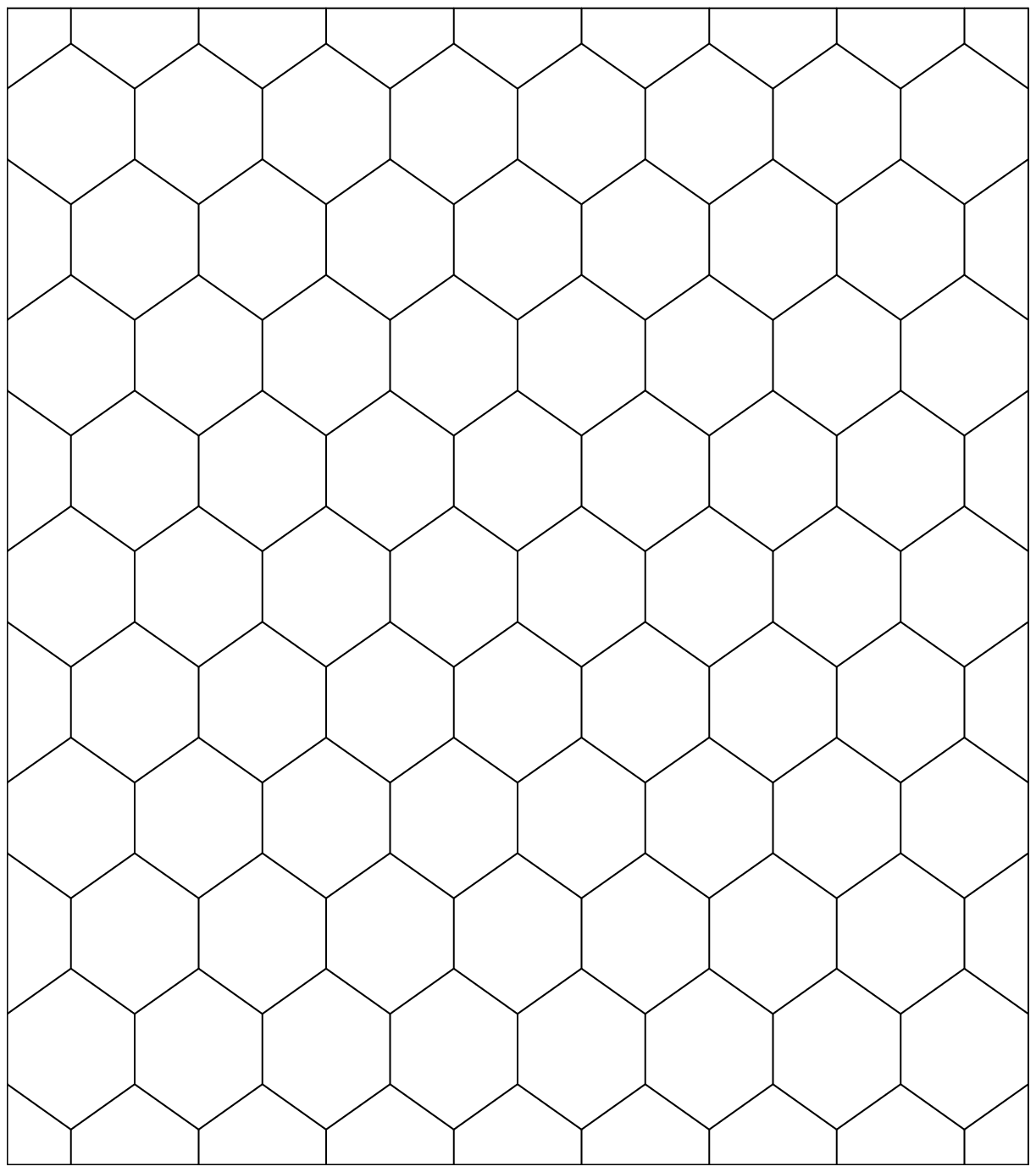}
\end{minipage}
\caption{Sample meshes: $\CT_h^1$ (left), $\CT_h^2$ (middle) and
$\CT_h^{3}$ (right). In all of them ${N=9}$.}
\label{FIG:VM35}
\end{center}
\end{figure}

Let us remark that for triangular and rectangular meshes like $\CT_h^1$
and $\CT_h^2$, respectively, the discrete spaces $\bV_h$ coincide with
those of the standard lowest-order Raviart-Thomas discretization.
However, the resulting discrete problems are not the same. In fact, the
matrices corresponding to the left-hand side of Problem~\ref{P3} also
coincide, but this does not happen with the matrices corresponding to
right-hand side.

We report in Table~\ref{TABLA:1} the scaled lowest eigenvalues
$\widehat{\l}_{hi}:=\l_{hi}/\pi^{2}$ computed with the method analyzed
in this paper. The table also includes estimated orders of convergence.
The exact eigenvalues are also reported in the last column to allow for
comparison.

\begin{table}[ht]
\begin{center}
\caption{Test 1. Computed lowest eigenvalues $\widehat{\l}_{hi}$, $1\le
i\le5$, on different meshes with $\sigma_E=1$.}
\begin{tabular}{|c|c|c|c|c|c|c|c|c|}
\hline
$\CT_h$   & $\widehat{\l}_{hi}$ & $N=19$ & $N=35$ & $N=53$ & $N=71$ & 
Order & $\l_i$ \\
\hline
          & $\widehat{\l}_{h1}$ & 0.8248 & 0.8259 & 0.8262 & 0.8263 & 
	  2.01  & 0.82645 \\
$\CT_h^1$ & $\widehat{\l}_{h2}$ & 0.9976 & 0.9993 & 0.9997 & 0.9998 & 
2.00  & 1.00000 \\
          & $\widehat{\l}_{h3}$ & 1.8182 & 1.8240 & 1.8254 & 1.8259 & 
	  2.01  & 1.82645 \\
          & $\widehat{\l}_{h4}$ & 3.2788 & 3.2978 & 3.3023 & 3.3039 & 
	  2.02  & 3.30579 \\
          & $\widehat{\l}_{h5}$ & 3.9595 & 3.9883 & 3.9949 & 3.9972 & 
	  2.03  & 4.00000 \\
\hline
          & $\widehat{\l}_{h1}$ & 0.8200 & 0.8245 & 0.8256 & 0.8260 & 
	  1.99  & 0.82645 \\
$\CT_h^2$ & $\widehat{\l}_{h2}$ & 0.9896 & 0.9969 & 0.9987 & 0.9992 & 
1.99  & 1.00000 \\
          & $\widehat{\l}_{h3}$ & 1.8096 & 1.8214 & 1.8243 & 1.8252 & 
	  1.99  & 1.82645 \\
          & $\widehat{\l}_{h4}$ & 3.2047 & 3.2754 & 3.2925 & 3.2983 & 
	  1.98  & 3.30579 \\
          & $\widehat{\l}_{h5}$ & 3.8389 & 3.9512 & 3.9786 & 3.9880 & 
	  1.97  & 4.00000 \\
\hline
          & $\widehat{\l}_{h1}$ & 0.8249 & 0.8260 & 0.8262 & 0.8263 & 
	  1.98  & 0.82645 \\
$\CT_h^3$ & $\widehat{\l}_{h2}$ & 0.9948 & 0.9982 & 0.9990 & 0.9993 & 
1.56  & 1.00000 \\
          & $\widehat{\l}_{h3}$ & 1.8132 & 1.8220 & 1.8241 & 1.8249 & 
	  1.63  & 1.82645 \\
          & $\widehat{\l}_{h4}$ & 3.2805 & 3.2979 & 3.3024 & 3.3039 & 
	  1.98  & 3.30579 \\
          & $\widehat{\l}_{h5}$ & 3.9387 & 3.9823 & 3.9912 & 3.9946 & 
	  1.84  & 4.00000 \\
\hline
\end{tabular}
\label{TABLA:1}
\end{center}
\end{table}

It can be seen from Table~\ref{TABLA:1} that the computed eigenvalues
converge to the exact ones with an optimal quadratic order as predicted
by the theory in almost all cases. The exception seems to be the
computation of some of the eigenvalues with the hexagonal meshes. In
this case, although the computed eigenvalues are as good approximations
to the exact ones as those computed with the other families of meshes,
the order of convergence deteriorates mildly. We have observed from our
numerical experiments that this can be avoided by choosing a smaller
stability constant $\sigma_E$. 

This can be clearly seen by comparing the lowest part of
Table~\ref{TABLA:1} with Table~\ref{TABLA:11-}, where we report the
result obtained with a smaller value of $\sigma_E$ and meshes $\CT_h^3$.
A more detailed discussion about the effect of the stability constant
$\sigma_E$ appears in the following test.

\begin{table}[ht]
\begin{center}
\caption{Test 1. Computed lowest eigenvalues $\widehat{\l}_{hi}$, $1\le
i\le5$, on meshes $\CT_h^3$ with $\sigma_E=2^{-4}$.}
\begin{tabular}{|c|c|c|c|c|c|c|c|c|}
\hline
$\CT_h$   & $\widehat{\l}_{hi}$ & $N=19$ & $N=35$ & $N=53$ & $N=71$ & 
Order & $\l_i$  \\
\hline
          & $\widehat{\l}_{h1}$ & 0.8294 & 0.8272 & 0.8268 & 0.8266 & 
	  2.09  & 0.82645 \\
$\CT_h^3$ & $\widehat{\l}_{h2}$ & 1.0032 & 1.0009 & 1.0004 & 1.0002 & 
2.15  & 1.00000 \\
          & $\widehat{\l}_{h3}$ & 1.8389 & 1.8297 & 1.8278 & 1.8272 & 
	  2.12  & 1.82645 \\
          & $\widehat{\l}_{h4}$ & 3.3539 & 3.3179 & 3.3112 & 3.3088 & 
	  2.09  & 3.30579 \\
          & $\widehat{\l}_{h5}$ & 4.0536 & 4.0149 & 4.0063 & 4.0034 & 
	  2.09  & 4.00000 \\
\hline
\end{tabular}
\label{TABLA:11-}
\end{center}
\end{table}

Figure~\ref{FIG:VM36} shows plots of the computed eigenfunctions
$\bw_{h1}$ and $\bw_{h3}$ corresponding to the first and third lowest
eigenvalues, respectively. The figure also includes the corresponding
pressure fluctuation $p_{hi}=-\div\bw_{hi}$, $i=1,3$. In both cases, the
eigenfunctions have been computed on an hexagonal mesh $\CT_h^3$ with
$N=27$ and stability constant $\sigma_E=1$.

\begin{figure}[ht]
\begin{center}
\begin{minipage}{6.3cm}
\centering\includegraphics[height=6.3cm, width=6.3cm]{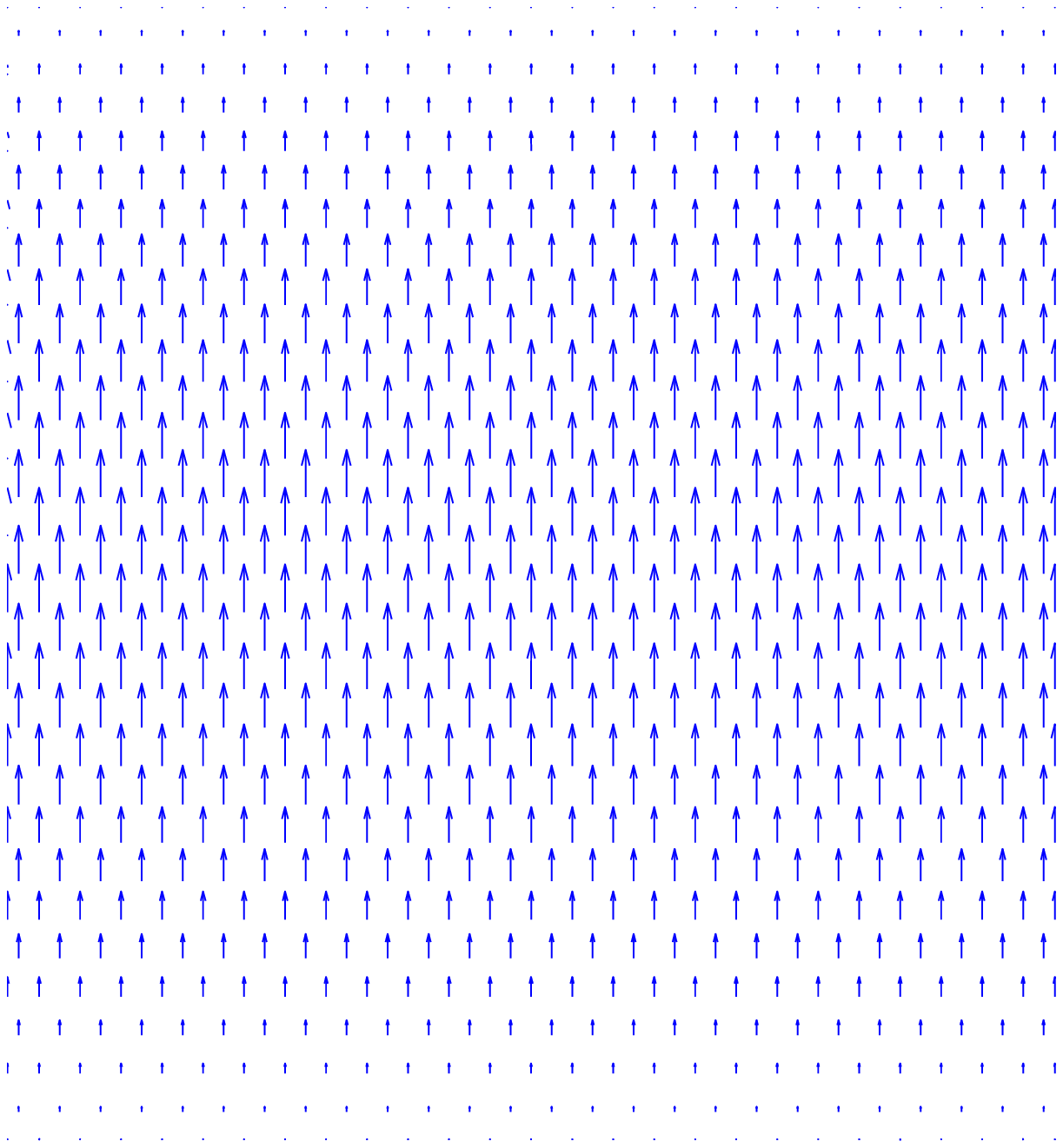}
\end{minipage}
\begin{minipage}{6.3cm}
\centering\includegraphics[height=6.3cm, width=6.3cm]{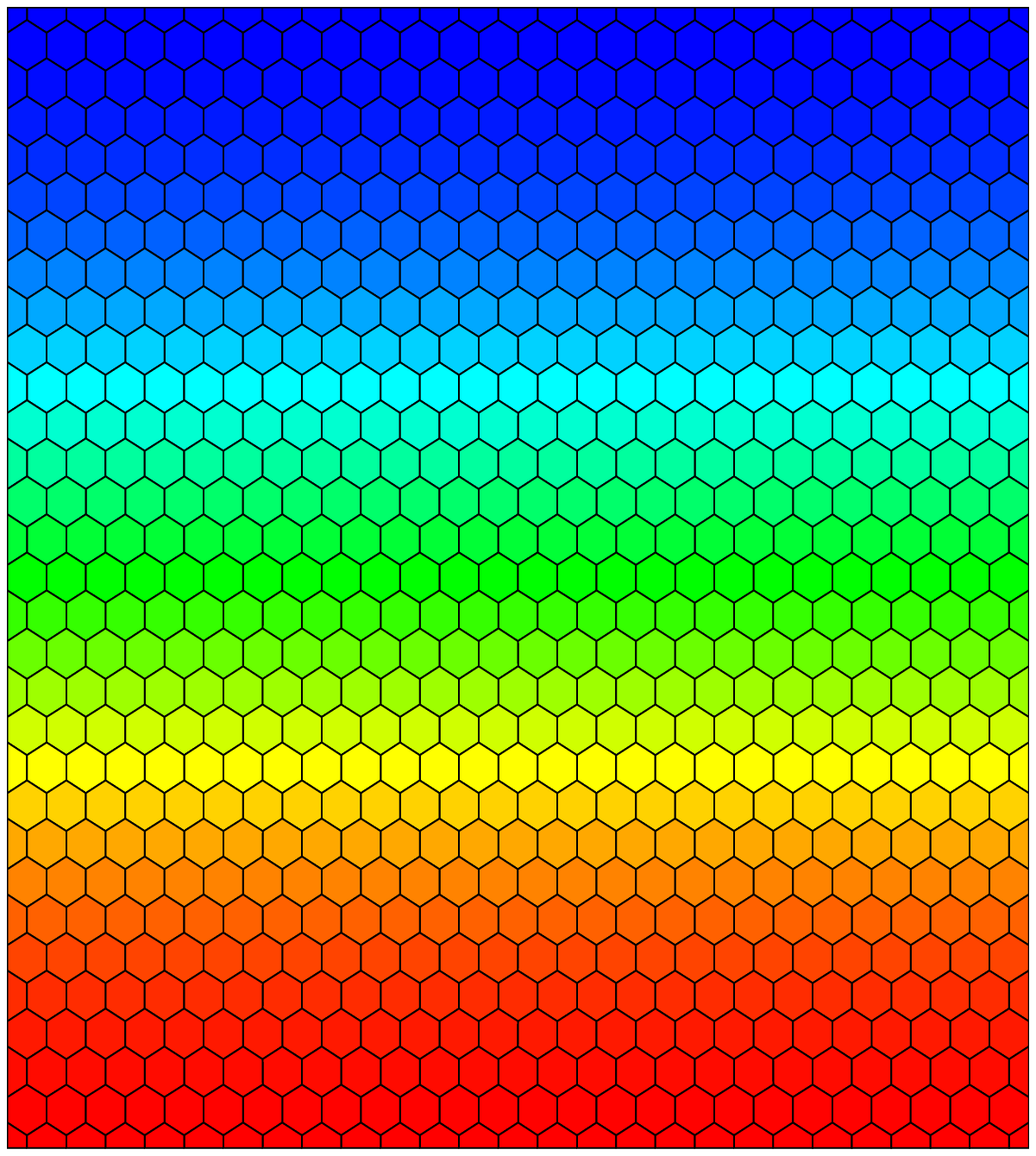}
\end{minipage}
\begin{minipage}{6.3cm}
\centering\includegraphics[height=6.3cm, width=6.3cm]{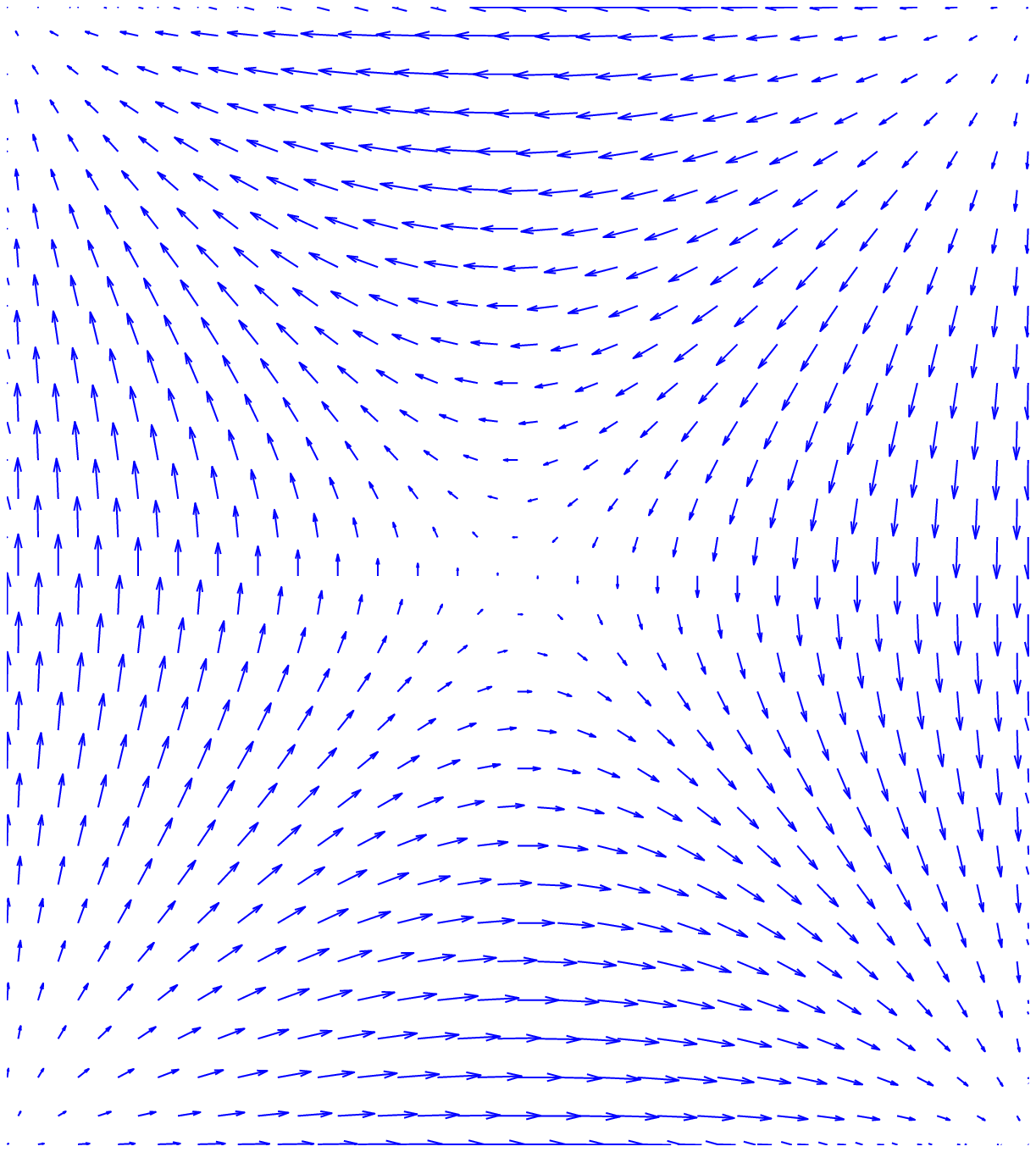}
\end{minipage}
\begin{minipage}{6.3cm}
\centering\includegraphics[height=6.3cm, width=6.3cm]{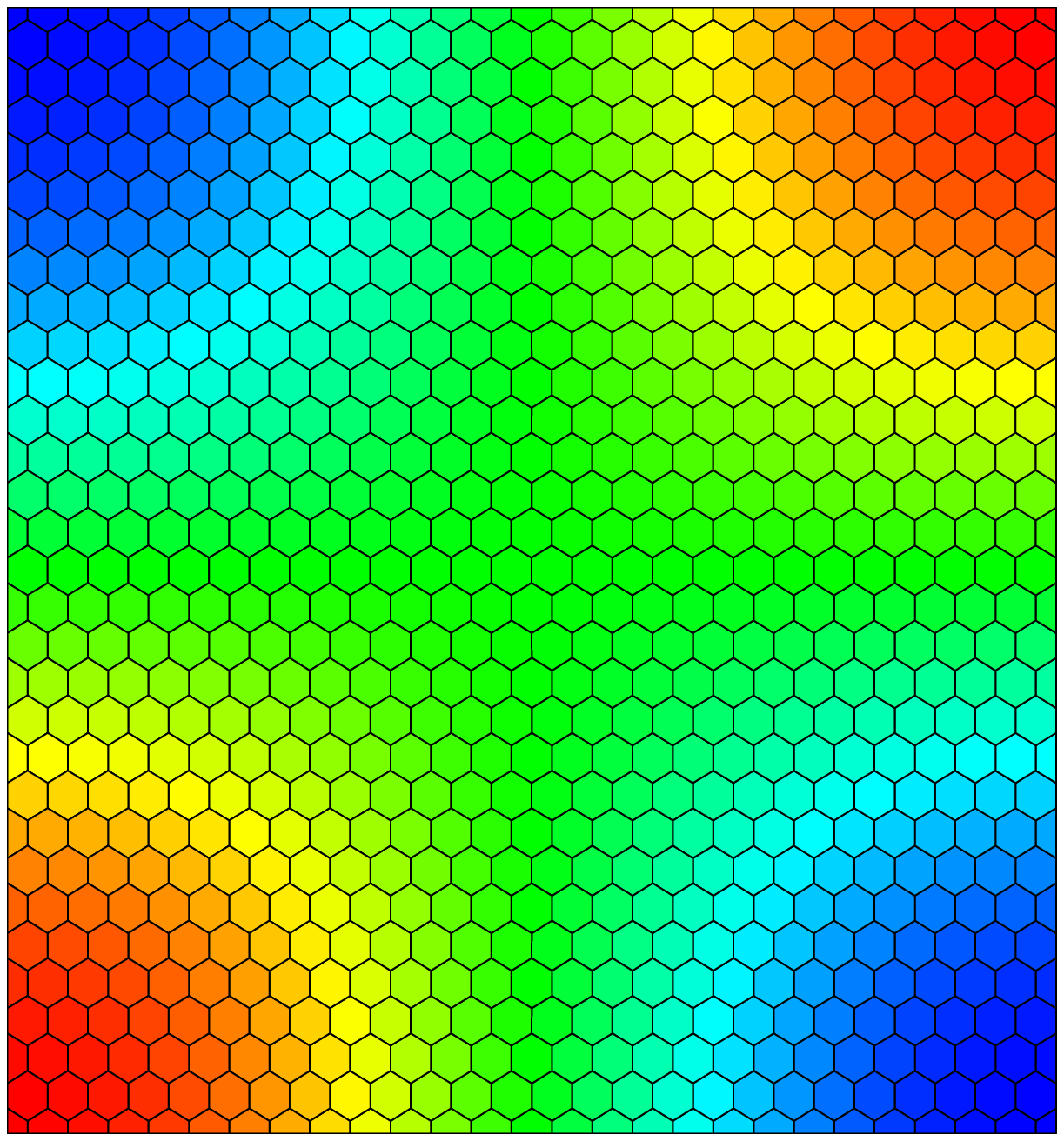}
\end{minipage}
\caption{Eigenfunctions of the acoustic problem corresponding to the
first and third lowest eigenvalues: displacement field $\bw_{h1}$ (upper
left), pressure fluctuation $p_{h1}$,(upper right), displacement field
$\bw_{h3}$ (bottom left), pressure fluctuation $p_{h3}$ (bottom right).}
\label{FIG:VM36}
\end{center}
\end{figure}

\subsection{Test 2: Effect of the stability constant $\sigma_E$}

As was shown in the previous test, in some cases the quality of the
computation can be affected by the choice of the stability constant
$\sigma_E$. A similar behavior was observed in other VEM for different
eigenvalue problems. In particular, it was demonstrated in
\cite{MRR2015} that certain VEM discretizations of the Steklov
eigenvalue problem introduces spurious eigenvalues which can be well
separated from the physical spectrum by choosing appropriately the
stability constant $\sigma_E$.

In the present case, no spurious eigenvalue was detected for any choice
of the stability constant. However, for large values of $\sigma_E$, the
eigenvalues computed with coarse meshes could be very poor. The aim of
this test is to analyze the influence of the stability constant
$\sigma_E$ on the computed spectrum.

We report in Table~\ref{TABLA:2} the lowest eigenvalue computed with
varying values of $\sigma_E$ on the family of meshes $\T_{h}^2$ (see
Figure 2, middle). The table also includes the estimated order of
convergence.

\begin{table}[ht]
\begin{center}
\caption{Test 2. The lowest eigenvalue $\widehat{\l}_{h1}$ for
$\sigma_E=0$ and $\sigma_E=2^{-k}$ with $-6\leq k\leq 6$.}
\begin{tabular}{|c|c|c|c|c|c|c|c|}
\hline
 $N$ & $\sigma_E=0$ & $\sigma_E=2^{-6}$ & $\sigma_E=2^{-5}$ & 
 $\sigma_E=2^{-4}$ & $\sigma_E=2^{-3}$ & $\sigma_E=2^{-2}$ & 
 $\sigma_E=2^{-1}$ \\
\hline
   8 & 0.8482 & 0.8472 & 0.8463 & 0.8444 & 0.8406 & 0.8332 & 0.8187 \\
  16 & 0.8318 & 0.8316 & 0.8313 & 0.8309 & 0.8300 & 0.8281 & 0.8245 \\
  32 & 0.8278 & 0.8277 & 0.8277 & 0.8275 & 0.8273 & 0.8269 & 0.8260 \\
  64 & 0.8268 & 0.8268 & 0.8268 & 0.8267 & 0.8267 & 0.8265 & 0.8263 \\
 128 & 0.8265 & 0.8265 & 0.8265 & 0.8265 & 0.8265 & 0.8265 & 0.8264 \\
 256 & 0.8265 & 0.8265 & 0.8265 & 0.8265 & 0.8265 & 0.8265 & 0.8264 \\
 \hline
\!\!Order\!\! & 2.00 & 2.00 & 2.00 & 2.00 & 2.00 & 2.00 & 2.00 \\
\hline
$\l_1$ & 0.82645 & 0.82645 & 0.82645 & 0.82645 & 0.82645 & 0.82645 & 
0.82645 \\
\hline
 $N$ & $\sigma_E=2^{0}$ & $\sigma_E=2^{1}$ & $\sigma_E=2^{2}$ & 
$\sigma_E=2^{3}$ & $\sigma_E=2^{4}$ & $\sigma_E=2^{5}$ & 
$\sigma_E=2^{6}$ \\
\hline
   8 & 0.7912 & 0.7415 & 0.6586 & 0.5383 & 0.3943 & 0.2569 & 0.1513 \\
  16 & 0.8174 & 0.8034 & 0.7770 & 0.7289 & 0.6487 & 0.5317 & 0.3907 \\
  32 & 0.8242 & 0.8206 & 0.8135 & 0.7997 & 0.7735 & 0.7258 & 0.6463 \\
  64 & 0.8259 & 0.8250 & 0.8233 & 0.8196 & 0.8125 & 0.7988 & 0.7726 \\
 128 & 0.8263 & 0.8261 & 0.8256 & 0.8247 & 0.8229 & 0.8193 & 0.8123 \\
 256 & 0.8264 & 0.8264 & 0.8262 & 0.8260 & 0.8256 & 0.8247 & 0.8229 \\
\hline
\!\!Order\!\! & 1.99 & 1.97 & 1.94 & 1.90 & 1.82 & 1.70 & 1.55 \\
\hline
$\l_1$ & 0.82645 & 0.82645 & 0.82645 & 0.82645 & 0.82645 & 0.82645 
&  0.82645 \\
\hline
\end{tabular}
\label{TABLA:2}
\end{center}
\end{table}

It can be seen from Table~\ref{TABLA:2} that for values of the
parameter $\sigma_E\leq 1$ the computed eigenvalues depend very mildly
on this parameter. Moreover, this dependence becomes weaker, as the mesh
is refined or $\sigma_E$ is taken smaller. In fact, it can be seen from
this table that even the value $\sigma_E=0$ yields very accurate
results, in spite of the fact that for such a value of the parameter the
stability estimate and hence most of the proofs of the theoretical
results do not hold. On the other hand, it can be seen from
Table~\ref{TABLA:2} that the numerical results depend much more
significantly on this parameter $\sigma_E$ when it is chosen larger. In
such a case, the results for coarse meshes are poorer and more refined
meshes are needed for the computed eigenvalues to lie close to the exact
ones.

This analysis suggests that the user of $\H(\div)$ VEM for this kind of
spectral problems has to be aware of the risk of degeneration of the
eigenvalues for certain values of the stability constant $\sigma_E$. The
way of minimizing this risk in this case is to take small values of
$\sigma_E$ (what ``small'' means in a real problem will of course depend
on the value of the physical constants).

\appendix 
\section{Appendix}

We derive in this appendix optimal approximation properties for the
$\H(\div)$ virtual elements with vanishing rotor introduced in
Section~\ref{SEC:Discrete}. The main goal of this appendix will be to
prove the error estimates stated in Lemmas~\ref{lemmainter} and
\ref{lemmainterV_I} for the $\bV_h$-interpolant defined by
\eqref{uno}--\eqref{dos}. Let us remark that these results could be
useful for other applications as well.

Our first result, whose proof is quite straightforward, is a commuting
diagram property and some consequences that follow from it. We recall
that $P_k$ denotes the $\LO$-orthogonal projection onto the subspace 
$\left\{q\in\LO: q|_E\in\bbP_k(E)\quad\forall E\in\CT_h\right\}$.

\vspace{.25cm}
\noindent\textbf{Lemma~\ref{lemmainter}\ }
\textit{Let $\bv\in\bV$ be such that $\bv\in[\HtO]^2$ with $t>1/2$. Let
$\bv_I\in\bV_h$ be its interpolant defined by \eqref{uno}--\eqref{dos}.
Then,
$$
\div\bv_I=P_k(\div\bv)\quad\text{ in }\O.
$$
Consequently, for all $E\in\CT_h$,
$\left\|\div\bv_I\right\|_{0,E}\leq\left\|\div\bv\right\|_{0,E}$ and, if
$\div\bv|_{E}\in\HrE$ with $r\geq 0$, then
\begin{equation*}
\left\|\div\bv-\div\bv_I\right\|_{0,E}\leq Ch_E^{\min\{r,k+1\}}\left|\div\bv\right|_{r,E}. 
\end{equation*}}

\begin{proof}
As a consequence of \eqref{uno}--\eqref{dos}, for every element $E$ and
for every $q\in\bbP_k(E)$
\begin{equation*}
\int_{E}\div(\bv-\bv_{I})\,q
=\int_{E}(\bv-\bv_{I})\cdot\nabla q
+\int_{\partial E}\left(\bv-\bv_{I}\right)\cdot\bn\,q\ds=0.
\end{equation*}
Since $\div\bv_{I}\in\bbP_k(E)$, we have that 
\begin{equation*}
\label{ext1}
\div\bv_{I}=P_k(\div\bv)\quad\text{in }E.
\end{equation*}
Therefore, 
$$
\left\|\div\bv_I\right\|_{0,E}
\leq\left\|\div\bv\right\|_{0,E}.
$$

Additionally, if $\div\bv|_{E}\in\HrE$ with $ r$ a non-negative integer,
as a consequence of \cite[Lemma~4.3.8]{BS-2008}, we have that for every
$E\in\CT_h$
\begin{equation*}
\label{estimdiv}
\left\|\div\bv-\div\bv_I\right\|_{0,E}
\leq Ch_E^{\min\{r,k+1\}}\left|\div\bv\right|_{r,E}. 
\end{equation*}
Thus, the second estimate of the lemma follows by standard Banach space
interpolation.
$\hfill\qed$
\end{proof}

In order to prove Lemma~\ref{lemmainterV_I} about the $\LO$
approximation property of this interpolant, we need several previous
results. We begin with the following local trace estimate on polygons.

\begin{lemma}
\label{trace_s}
Let $\bv\in\bV$ and $E\in\CT_{h}$ such that $\bv|_{E}\in[\HtE]^2$ with
$t\in(1/2,1]$. Then, there exists $C>0$ such that
\begin{equation*}
\label{tracelocal_s}
\left\|\bv\right\|_{0,\partial E}
\leq 
C\left(h_E^{-1/2}\left\|\bv\right\|_{0,E}
+h_E^{t-1/2}\left|\bv\right|_{t,E}\right).
\end{equation*}
\end{lemma}

\begin{proof}
Consider the triangulation $\CT_h^{E}$ of the element $E$ obtained by
joining each vertex of $E$ with the midpoint of the ball with respect to
which $E$ is star-shaped. Since we are assuming that the meshes satisfy
$\mathbf{A_1}$ and $\mathbf{A_2}$, the triangles $T\in\CT_h^{E}$ have a
shape ratio (i.e., the quotient between outer and inner diameters)
bounded above by a constant that only depends on $C_\CT$. Moreover, each
triangle $T\in\CT_h^{E}$ has one edge on $\partial E$. Hence, a scaling
argument and a trace inequality in the reference triangular element
allow us to conclude the proof.
$\hfill\qed$
\end{proof}

In order to prove an $\LO$ error estimate for the interpolant $\bv_{I}$,
we will introduce a basis of $\bV_h^E$ dual to the degrees of freedom
\eqref{freedom}--\eqref{freedom2}.

Let $E\in\T_{h}$ with edges $e_1,\ldots,e_{N_E}$ and
$F:E\longrightarrow\widehat{E}$ be an affine mapping of the form
$$
F\begin{pmatrix} x \\ y \end{pmatrix} 
:=\dfrac{1}{h_E}\begin{pmatrix} x-x_E \\ y-y_E \end{pmatrix} 
=:\begin{pmatrix} \widehat{x} \\ \widehat{y} \end{pmatrix},
$$ 
where $\bx_{E}=(x_E,y_E)^T$ is the center of the ball with respect to
which $E$ is star-shaped according to assumption $\mathbf{A_2}$. Note
that $\widehat{E}:=F(E)$ has diameter 1. Moreover, $F$ maps the above
mentioned ball onto a ball of radius $C_{\CT}$ with $0<C_{\CT}\leq 1 $
and $C_{\CT}$ independent of $h_E$, Moreover, $\widehat{E}$ is
star-shaped with respect to each point of this ball.

We define the following basis of $\bbP_k(E):$
\begin{align*}
p_{0}(x,y) & :=1,
\\
p_{s}(x,y) 
& :=\dfrac{\left(x-x_{E}\right)^{\alpha_{1}}
\left(y-y_{E}\right)^{\alpha_{2}}}
{h_{E}^{\alpha_{1}+\alpha_{2}}}
+C_{s},
\qquad\alpha_{1},\alpha_{2}\in\N,
\quad0<\alpha_{1}+\alpha_{2}\leq k,
\end{align*}
with the constant $C_{s}\in\R$ such that $\int_{E}p_{s}=0$. We have
associated above each $s=1,\ldots,\widetilde{N}:=\dim(\bbP_k(E))-1$ with
one particular couple $(\alpha_{1},\alpha_{2})$, by fixing a particular
ordering of these couples. Therefore, the set
$\left\{p_0,p_1,\ldots,p_{\widetilde{N}}\right\}$ is a basis for
$\bbP_k(E)$ that satisfies $\int_Ep_s=0$ for $
s=1,\ldots,\widetilde{N}$. Let now $\widehat{p}_{s}:=p_{s}\circ F^{-1}$
be defined in $\widehat{E}$. Then, for the particular
$(\alpha_{1},\alpha_{2})$ associated with $s$, we have that
$\widehat{p}_{s}(\widehat{x},\widehat{y})
=\widehat{x}^{\alpha_{1}}\widehat{y}^{\alpha_{2}}+C_{s}$. Moreover,
since $\left|E\right|=h_E^2\big|\widehat{E}\big|$, we have
$$
C_s=-\dfrac{1}{\left|E\right|}\int_{E}
\dfrac{\left(x-x_{E}\right)^{\alpha_{1}}
\left(y-y_{E}\right)^{\alpha_{2}}}
{h_{E}^{\alpha_{1}+\alpha_{2}}}\,dx\,dy
=-\dfrac{1}{\big|\widehat{E}\big|}\int_{\widehat{E}}
\widehat{x}^{\alpha_1}\widehat{y}^{\alpha_2}
\,d\widehat{x}\,d\widehat{y}.
$$
As a consequence, note that $\left|C_s\right|\leq 1$ and, hence,
$\left\|p_s\right\|_{\infty,E}
=\left\|\widehat{p}_s\right\|_{\infty,\widehat{E}}\leq2$, 
$s=0,\ldots,\widetilde{N}$.

For each edge $e_l$ of $E$ ($l=1,\ldots,N_E$), let $T_{l}$ be the affine
function mapping $\widehat{e}:=[-1,1]$ onto $e_l$. We define
$q_l^i:=\widehat{q}\,^{i}\circ T_{l}^{-1}$ ($i=1,\dots,k$) with
$\widehat{q}\,^{i}$ being the Legendre polynomials on $[-1,1]$
normalized by $\widehat{q}\,^{i}(1)=1$. Then,
$\left\{q_l^0,\ldots,q_{l}^k\right\}$ is a basis of $\bbP_k(e_l)$ which
satisfies $q_l^0=1$, $\int_{e_l} q_l^iq_l^j\ds=\delta_{ij}$, $
i,j=1,\ldots,k$, and $\left\|q_l^i\right\|_{\infty,e_l}=1$. Note that,
in particular, $\int_{e_l} q_l^i\ds=0$, $ i=1,\ldots,k$.

Therefore,
\begin{equation*}
\left\{q_{l}^{i}\right\}_{i=0,\ldots,k,\ l=1,\ldots,N_E}
\qquad\textrm{and}\qquad
\left\{p_{s}\right\}_{s=1,\ldots,\widetilde{N}}
\end{equation*}
are bases for the spaces of test functions appearing in the degrees of
freedom \eqref{uno} and \eqref{dos}, respectively. Next, we introduce a
set of dual basis functions for $\bV_h^E$:
\begin{equation}
\label{spanv}
\left\{\bphi_l^i\right\}_{i=0,\ldots,k,\ l=1,\ldots,N_E}
\cup\left\{\widetilde{\bphi}^s\right\}_{s=1,\ldots,\widetilde{N}}.
\end{equation}
The first ones, $\bphi_{l}^i$, are the ``boundary basis functions''
determined by
\begin{align}
\label{primera0}
& \bphi_l^i\in\bV_{h}^{E},
\\
\label{primera}
& \int_{e_{m}}\left(\bphi_l^i\cdot\bn\right)q_{m}^j\ds
=\delta_{lm}\delta_{ij},
\qquad m=1,\ldots,N_{E},\quad j=0,\ldots,k,
\\
\label{segunda}
& \int_{E}\left(\div\bphi_l^i\right)p_{r}=0,
\qquad r=1,\ldots,\widetilde{N}.
\end{align}
Note that these boundary basis functions use two indexes, $i$ and $l$,
one for the moment and the other for the edge. On the other hand, note
also that as a consequence of \eqref{primera0}--\eqref{primera}
$\bphi_l^i\cdot\bn=0$ on $\partial E\setminus e_{l}$ The second kind of
functions in \eqref{spanv}, $\widetilde{\bphi}^s$, are the ``internal
basis functions'' determined by
\begin{align}
\label{tercera0}
& \widetilde{\bphi}^s\in\bV_{h}^{E},
\\
\label{tercera}
& \widetilde{\bphi}^s|_{\partial E}\cdot\bn=0,
\\
\label{cuarta}
& \int_{E}\left(\div\widetilde{\bphi}^s\right)p_{r}
=\delta_{sr},
\qquad r=1,\ldots,\widetilde{N}.
\end{align}

\begin{remark}
Since $\div\bphi_{l}^i\in\bbP_k(E)
=\mathrm{span}\left\{1,p_1,\ldots,p_s\right\}$ and $\int_Ep_{s}=0$ for
$s=1,\ldots,\widetilde{N}$, equation \eqref{segunda} implies that
$\div\bphi_{l}^i$ has to be constant. Therefore,
\begin{equation*}
\div\bphi_{l}^i
=\dfrac{1}{\left|E\right|}\int_{E}\div\bphi_{l}^i
=\dfrac{1}{\left|E\right|}\int_{\partial E}\bphi_{l}^i\cdot\bn\ds.
\end{equation*} 
Moreover, thanks to \eqref{primera}, we have that
\begin{equation*}
\int_{\partial E}\bphi_{l}^i\cdot\bn\ds
=\sum_{m=1}^{N_{E}}\int_{e_{m}}\left(\bphi_l^i\cdot\bn\right)q_{m}^0\ds
=\sum_{m=1}^{N_{E}}\delta_{l m}\delta_{i0}=\delta_{i0}.
\end{equation*}
Then, 
\begin{equation*}
\div\bphi_{l}^i
=\dfrac{\delta_{i0}}{\left|E\right|}.
\end{equation*}
\end{remark}
Next goal is to prove that all the functions in \eqref{spanv} are
bounded uniformly in $h$. We begin with the boundary basis functions.

\begin{lemma}
\label{cotatheta}
There exists $C>0$ such that $\left\|\bphi_{l}^i\right\|_{0,E}\leq C$
for $l=1,\ldots,N_{E}$ and $i=0,\ldots,k$.
\end{lemma}

\begin{proof}
Since $\bphi_{l}^i\in\bV_h^E$, we know that $\rot\bphi_{l}^i=0$.
Therefore, there exists $\gamma\in\HuE$ such that
$\bphi_{l}^i=\nabla\gamma$. Hence, from the remark above and
\eqref{primera}, we have that $\gamma$ is a solution of the following
problem:
\begin{align*}
\left\{\begin{array}{ll}
\Delta\gamma
=\dfrac{\delta_{i0}}{\left|E\right|}
& \quad\text{in }E,
\\[0.15cm]
\dfrac{\partial\gamma}{\partial\bn}
=\bphi_{l}^i\cdot\bn
& \quad\text{on }\partial E,
\\[0.25cm]
\displaystyle\int_E\gamma=0.
\end{array}\right.
\end{align*}
It is easy to check that these Neumann problems are compatible.
Therefore, 
\begin{equation*}
\int_{E}\nabla\gamma\cdot\nabla 
\zeta=\int_{\partial E}\left(\bphi_{l}^i\cdot\bn\right)\zeta\ds
-\int_E\dfrac{\delta_{i0}}{\left|E\right|}\zeta
=\int_{e_l}\left(\bphi_{l}^i\cdot\bn\right)\zeta\ds 
\qquad\forall\zeta\in\HuE:\ \int_{E}\zeta=0.
\end{equation*}
Now, taking $\zeta=\gamma$, we obtain
\begin{align*}
\left\|\bphi_{l}^i\right\|_{0,E}^2
& =\left\|\nabla\gamma\right\|_{0,E}^2
\leq\left\|\bphi_{l}^i\cdot\bn\right\|_{0,e_l}
\left\|\gamma\right\|_{0,e_l}
\\
& \leq C\left\|\bphi_{l}^i\cdot\bn\right\|_{0,e_l}
\left(h_E^{-1/2}\left\|\gamma\right\|_{0,E}
+h_E^{1/2}\left\|\nabla\gamma\right\|_{0,E}\right) 
\\
& \leq C h_E^{1/2}\left\|\bphi_{l}^i\cdot\bn\right\|_{0,e_l}
\left\|\nabla\gamma\right\|_{0,E},
\end{align*}
where we have used Lemma~\ref{trace_s} with $t=1$, the generalized
Poincar\'e inequality and a scaling argument. Now, because of
\eqref{primera} with $m=l$ and the orthogonality property of Legendre
polynomials, $\left.\bphi_l^i\cdot\bn\right|_{e_l}
=\left(\int_{e_l}\left(q^{i}_l\right)^2\ds\right)^{-1}q_l^i$.
Therefore,
\begin{equation*}
\left\|\bphi_{l}^i\cdot\bn\right\|_{0,e_l}^2
=\left(\int_{e_l}\left(q^i_l\right)^2\ds\right)^{-1}
=\dfrac{1}{h_E}\left(\int_{\widehat{e}}
\left(\widehat{q}\,^i\right)^2\,d\widehat{s}\right)^{-1}.
\end{equation*}
Thus, from the last two estimates we derive that
$\left\|\bphi_{l}^i\right\|_{0,E}\leq C$ and we end the proof.
$\hfill\qed$
\end{proof}

Next, we show a similar result for the internal basis functions.

\begin{lemma}
\label{cotatheta2}
There exists $C>0$ such that
$\left\|\widetilde{\bphi}^s\right\|_{0,E}\leq C$ for
$s=1,\ldots,\widetilde{N}$.
\end{lemma}

\begin{proof}
Since $\widetilde{\bphi}^s\in\bV_h^E$, there exists $\gamma\in\HuE$ such
that $\widetilde{\bphi}^s=\nabla\gamma$. Hence, by virtue of
\eqref{tercera}, we have that $\gamma$ is a solution of the following
well posed Neumann problem:
\begin{align*}
\left\{\begin{array}{ll}
\Delta\gamma=-\div\widetilde{\bphi}^s
& \quad\text{in }E,
\\[0.15cm]
\dfrac{\partial\gamma}{\partial\bn}=0 
& \quad\text{on }\partial E,
\\[0.25cm]
\displaystyle\int_E\gamma=0.
\end{array}\right.
\end{align*}
Therefore, 
\begin{equation*}
\int_{E}\nabla\gamma\cdot\nabla\zeta
=-\int_{E}\psi^s\zeta
\qquad\forall\zeta\in\HuE:\ \int_{E}\zeta=0,
\end{equation*}
where $\psi^s:=\div\widetilde{\bphi}^s$. Now, taking $\zeta=\gamma$ and
using the generalized Poincar\'e inequality and a scaling argument, we
have that
$$
\left\|\widetilde{\bphi}^s\right\|_{0,E}^2
=\left\|\nabla\gamma\right\|_{0,E}^2
\leq C\left\|\psi^s\right\|_{0,E}
\left\|\gamma\right\|_{0,E}
\leq Ch_E\left\|\psi^s\right\|_{0,E}
\left\|\nabla\gamma\right\|_{0,E}.
$$
Thus,
\begin{equation}
\label{rodol}
\left\|\widetilde{\bphi}^s\right\|_{0,E}
\leq C h_E\left\|\psi^s\right\|_{0,E}.
\end{equation}

On the other hand, since $\psi^s\in\bbP_k(E)$, it is easy to
check that
\begin{equation}
\label{difhue2}
h_{E}\left\|\psi^{s}\right\|_{0,E}
\leq Ch_{E}^{2}\left\|\psi^{s}\right\|_{\infty,E}
=Ch_{E}^{2}\,\big\|\widehat{\psi}^{s}\big\|_{\infty,\widehat{E}},
\end{equation}
where $\widehat{\psi}^{s}
:=(\psi^{s}\circ F^{-1})\in\bbP_k(\widehat{E})$.

For $\widehat{\psi}^{s}\in\bbP_k(\widehat{E})$, we write
$\widehat{\psi}^{s}=\sum_{i=1}^{\widetilde{N}}\beta_i^s\widehat{p}_i$
and, since $\left\|\widehat{p}_{i}\right\|_{\infty,\widehat{E}}\leq 2$,
we have that
\begin{equation}
\label{sino}
\big\|\widehat{\psi}^{s}\big\|_{\infty,\widehat{E}}
\leq\max_{1\leq i\leq\widetilde{N}}
\left|\beta_{i}^s\right|
\sum_{i=1}^{\widetilde{N}}
\left\|\widehat{p}_{i}\right\|_{\infty,\widehat{E}}
\leq C\max_{1\leq i\leq\widetilde{N}}
\left|\beta_{i}^s\right|.
\end{equation}
Now, from \eqref{cuarta}, a change of variables from $E$ to
$\widehat{E}$ yields 
$$
\int_{\widehat{E}}\widehat{\psi}^{s}\widehat{p}_{r}
=h_{E}^{-2}\delta_{sr},
\qquad r=1,\ldots,\widetilde{N},
$$
which can be written as
\begin{equation}
\label{aster}
\sum_{i=1}^{\widetilde{N}}
\beta_i^s\int_{\widehat{E}}\widehat{p}_i\widehat{p}_{r}
=h_{E}^{-2}\delta_{sr},
\qquad r=1,\ldots,\widetilde{N}.
\end{equation}
Let 
\begin{equation*}
\label{rodo2}
\bM=\begin{pmatrix}m_{ir}\end{pmatrix}
\in\R^{\widetilde{N}\times\widetilde{N}}
\qquad\text{with }
m_{ir}:=\int_{\widehat{E}}\widehat{p}_{i}\widehat{p}_{r},
\quad i,r=1,\ldots,\widetilde{N}.
\end{equation*}
Therefore, from \eqref{aster}, if $\bM$ is invertible, then
$\boldsymbol{\beta}^s=\begin{pmatrix}\beta_1^s & \cdots &
\beta_{\widetilde{N}}^s\end{pmatrix}^T$ is equal to $h_{E}^{-2}$ times
the $s$-th column of $\bM^{-1}$. 

Next, we will show that $\bM$ is invertible and that its inverse is
bounded uniformly in $h$. With this aim, note that the polygon
$\widehat{E}$ is uniquely defined by the vector
$((\widehat{x}_1,\widehat{y}_1),\ldots,
(\widehat{x}_{N_E},\widehat{y}_{N_E}))\in\R^{2N_{E}}$ that collects the
coordinates of its (ordered) vertexes. Let $U\subset\R^{2N_{E}}$, be the
set of all possible values of these coordinates such that the mesh
regularity assumptions $\mathbf{A_1}$ and $\mathbf{A_2}$ are satisfied.
Since the diameter of $\widehat{E}$ is equal to 1, $U$ is a bounded set.
On the other hand, the constraints that arise from hypotheses
$\mathbf{A_1}$ and $\mathbf{A_2}$ yield that $U$ is a closed set.
Therefore $U$ is compact.

The function from $U$ into $\R^{\widetilde{N}\times\widetilde{N}}$ that
maps the coordinates of the vertexes of $\widehat{E}$ into the entries
of the matrix $\bM$ is a continuous function. Moreover, for any
coordinates in $U$, $\widehat{E}$ satisfies $\mathbf{A_1}$ and
$\mathbf{A_2}$ and, hence, it contains a ball of radius $C_\CT$. Let us
show that this implies that $\bM$ has to be positive definite. In fact,
given $\alpha\in\R^{\widetilde{N}}$,
$\alpha^{T}\bM\alpha=\int_{\widehat{E}}
\big|\sum_{r=1}^{\widetilde{N}}\alpha_{r}\widehat{p}_{r}\big|^{2} \geq0$
and the equality holds only if
$\sum_{r=1}^{\widetilde{N}}\alpha_{r}\widehat{p}_{r}$ vanishes a.e. in
$\widehat{E}$, which in turn implies that $\alpha$ has to vanish (since
$\widehat{E}$ contains a ball of radius $C_\CT>0$). Thus, $\bM$ is
positive definite and hence invertible. Therefore, taking also into
account the continuity of the mapping $\bM\longmapsto\bM^{-1}$ for
invertible matrices, we conclude that the mapping 
$$
U\ni((\widehat{x}_1,\widehat{y}_1),\ldots,
(\widehat{x}_{N_E},\widehat{y}_{N_E}))
\longmapsto\bM^{-1}\in\R^{\widetilde{N}\times\widetilde{N}}
$$
is well defined and continuous and, hence, bounded above in the compact
set $U$. Consequently, from \eqref{aster},
\begin{equation*}
\left\|\boldsymbol{\beta}^s\right\|_{\infty}\leq Ch_E^{-2},
\end{equation*}
which recalling \eqref{sino} yields
\begin{equation}
\label{matrix}
\big\|\widehat{\psi}^{s}\big\|_{\infty,\widehat{E}}\leq Ch_E^{-2}.
\end{equation}
Let us remark that, in principle, the constant $C$ above depends on the
number $N_E$ of vertexes of $E$. However, by virtue of assumption
$\mathbf{A_1}$, this number is bounded above in terms of $C_{\T}$.
Therefore, $N_E$ can take only a finite number of possible values and,
hence, \eqref{matrix} holds true with $C$ only depending on $C_{\T}$.
Thus, we conclude the proof by combining \eqref{rodol}, \eqref{difhue2}
and \eqref{matrix}.
$\hfill\qed$
\end{proof}

Now, we are in a position to prove $\LO$ error estimates for the
$\bV_h$-interpolant.

\vspace{0.35cm}
\noindent\textbf{Lemma~\ref{lemmainterV_I}\ }
\textit{Let $\bv\in\bV$ be such that $\bv\in[\HtO]^2$ with $ t >1/2$.
Let $\bv_I\in\bV_h$ be its interpolant defined by
\eqref{uno}--\eqref{dos}. Let $E\in\CT_h$. If $1\leq t\leq k+1$, then
\begin{equation}
\label{estimatevII}
\left\|\bv-\bv_I\right\|_{0,E}\leq Ch_E^{t}\left|\bv\right|_{t,E},
\end{equation}
whereas, if $1/2 <t\leq1$, then
\begin{equation}
\label{estimatev2II}
\left\|\bv-\bv_I\right\|_{0,E}
\leq C\left(h_E^{t}\left|\bv\right|_{t,E}
+h_E\left\|\div\bv\right\|_{0,E}\right).
\end{equation}}

\begin{proof}
First, we consider the case $1\leq t\leq k+1$. The first step is to
bound $\left\|\bv_I\right\|_{0,E}$. Since $\bv_{I}\in\bV_{h}^{E}$,
thanks to \eqref{primera0}--\eqref{cuarta} we write it in the basis 
\eqref{spanv} as follows:
\begin{equation*}
\bv_I=\sum_{l=1}^{N_E}\sum_{i=0}^{k}
\left(\int_{e_l}\left(\bv\cdot\bn\right)q_{l}^i\ds\right)\bphi_l^{i}
+\sum_{s=1}^{\widetilde{N}}
\left(\int_{E}\left(\div\bv\right)p_{s}\right)\widetilde{\bphi}^{s}.
\end{equation*}
Therefore, from Lemmas~\ref{cotatheta} and \ref{cotatheta2} we have
$$
\left\|\bv_I\right\|_{0,E}
\leq C\left(\sum_{l=1}^{N_E}\sum_{i=0}^{k}
\left|\int_{e_l}\left(\bv\cdot\bn\right)q_{l}^i\ds\right|
+\sum_{s=1}^{\widetilde{N}}
\left|\int_{E}\left(\div\bv\right)p_{s}\right|\right).
$$
Then, by using that $\left\|q_{l}^{i}\right\|_{\infty,e_{l}}\!\!\!=1$
for $i=1,\dots,k$ and $l=1,\ldots,N_{E}$,
$\left\|p_{s}\right\|_{\infty,E}\leq C$ for $s=1,\dots,\widetilde{N}$,
the Cauchy-Schwarz inequality and Lemma~\ref{trace_s}, we obtain
\begin{align}
\nonumber
\left\|\bv_I\right\|_{0,E}
& \leq C\left(h_E^{1/2}\left\|\bv\right\|_{0,\partial E}
\left\|q_{l}^i\right\|_{\infty,e_l}
+\widetilde{N}h_E\left\|\div\bv\right\|_{0,E}
\left\|p_{s}\right\|_{\infty,E}\right)
\\
\nonumber
& \leq C\left(\left\|\bv\right\|_{0,E}
+h_E\left|\bv\right|_{1,E}
+h_{E}\left\|\div\bv\right\|_{0,E}\right)
\\
&\leq C\left(\left\|\bv\right\|_{0,E}
+h_E\left|\bv\right|_{1,E}\right).
\label{otraecuac}
\end{align}
Now, for all $\bv_k\in[\bbP_k(E)]^2$ we note that $(\bv_k)_I=\bv_k$ and,
hence, using the above estimate for $\bv-\bv_{k}$, we write
$$
\left\|\bv-\bv_I\right\|_{0,E}
=\left\|\bv-\bv_k-\left(\bv-\bv_k\right)_I\right\|_{0,E}
\leq\left\|\bv-\bv_k\right\|_{0,E}
+C\left(\left\|\bv-\bv_k\right\|_{0,E}
+h_E\left|\bv-\bv_k\right|_{1,E}\right).
$$
Thus, by choosing $\bv_{k}$ as in \cite[Proposition~4.2]{BBCMMR2013}, we
have that $\left\|\bv-\bv_k\right\|_{0,E}
+h_E\left|\bv-\bv_k\right|_{1,E}\leq C h_E^t\left|\bv\right|_{t,E}$,
which together with the above inequality allow us to conclude
\eqref{estimatevII}.

Next, we consider the case $1/2 <t\le1$. Using the same arguments as
above, we obtain in this case instead of \eqref{otraecuac},
\begin{equation}
\label{rodo3}
\left\|\bv_I\right\|_{0,E}
\leq C\left(\left\|\bv\right\|_{0,E}
+h_E^{t}\left|\bv\right|_{t,E}
+h_{E}\left\|\div\bv\right\|_{0,E}\right).
\end{equation}
Therefore, repeating again the arguments above with
$\bv_0\in[\bbP_0(E)]^2$ instead of $\bv_{k}$, we have
\begin{align*}
\left\|\bv-\bv_I\right\|_{0,E}
& \leq\left\|\bv-\bv_0-\left(\bv-\bv_0\right)_I\right\|_{0,E}
\\
& \leq\left\|\bv-\bv_0\right\|_{0,E}
+C\left(\left\|\bv-\bv_0\right\|_{0,E}
+h_E^t\left|\bv\right|_{t,E}
+h_E\left\|\div\bv\right\|_{0,E}\right)
\\
& \leq C\left(h_E^t\left|\bv\right|_{t,E}
+h_E\left\|\div\bv\right\|_{0,E}\right),
\end{align*}
where we have used again \cite[Proposition~4.2]{BBCMMR2013}. Thus, the
proof is complete.
$\hfill\qed$
\end{proof}

\begin{remark}
Estimate \eqref{estimatev2II} can be improved for $k=0$ and $1/2<t\le1$.
In fact, in such a case, the interpolant $\bv_I\in\bV_h$ is defined only
by \eqref{uno}.
Hence,
\begin{equation*}
\bv_I=\sum_{l=1}^{N_E}\left(\int_{e_l}
\left(\bv\cdot\bn\right)q_{l}^0\,ds\right)\bphi_l^{0}
\end{equation*}
and repeating the arguments above we obtain
\begin{equation*}
\left\|\bv_I\right\|_{0,E}
\leq C\left(\left\|\bv\right\|_{0,E}
+h_E^{t}\left|\bv\right|_{t,E}\right),
\end{equation*}
instead of \eqref{rodo3}, which leads to
$$
\left\|\bv-\bv_I\right\|_{0,E}
\leq C h_E^t\left|\bv\right|_{t,E}.
$$
\end{remark}

\bibliographystyle{amsplain}

\begin{thebibliography}{99}
\bibitem{equiv} 
Ahmad, B., Alsaedi, A., Brezzi, F., Marini, L.~D., Russo, A.:
Equivalent projectors for virtual element methods.
Comput. Math. Appl. \textbf{66}, 376--391 (2013).

\bibitem{ABMV2014} 
Antonietti, P.~F., Beir\~ao da Veiga, L., Mora, D., Verani, M.:
A stream virtual element formulation of the Stokes problem on polygonal
meshes.
SIAM J. Numer. Anal. \textbf{52}, 386--404 (2014).

\bibitem{AHSV} 
Antonietti, P.~F., Houston, P., Sarti, M., Verani, M.:
Multigrid algorithms for $hp$-version interior penalty discontinuous
Galerkin methods on polygonal and polyhedral meshes.
Preprint arXiv:1412.0913 [math.NA] (2014).

\bibitem{ALM15} 
Ayuso de Dios, B., Lipnikov, K., Manzini, G.:
The nonconforming virtual element method:
Preprint arXiv:1405.3741 [math.NA] (2014).

\bibitem{BO} 
Babu\v{s}ka, I., Osborn, J.: 
Eigenvalue problems 
In: Ciarlet, P.~G., Lions, J.~L. (eds.) 
Handbook of Numerical Analysis, 
Vol. II, pp. 641--787.
North-Holland, Amsterdam (1991).

\bibitem{BBCMMR2013} 
Beir\~ao da Veiga, L., Brezzi, F., Cangiani, A., Manzini, G., Marini,
L.~D., Russo, A.:
Basic principles of virtual element methods.
Math. Models Methods Appl. Sci. \textbf{23}, 199--214 (2013).

\bibitem{BBMR2014} 
Beir\~ao da Veiga, L., Brezzi, F., Marini, L.~D., Russo, A.:
The hitchhiker's guide to the virtual element method.
Math. Models Methods Appl. Sci. \textbf{24}, 1541--1573 (2014).

\bibitem{BBMR2015} 
Beir\~ao da Veiga, L., Brezzi, F., Marini, L.~D., Russo, A.:
Mixed virtual element methods for general second order elliptic problems
on polygonal meshes.
Preprint arXiv:1506.07328 [math.NA] (2015).

\bibitem{BLMbook2014} 
Beir\~ao da Veiga, L., Lipnikov, K., Manzini, G.:
The Mimetic Finite Difference Method for Elliptic Problems.
Springer, MS\&A, vol. 11 (2014).

\bibitem{BLM2015} 
Beir\~ao da Veiga, L., Lovadina, C., Mora, D.:
A virtual element method for elastic and inelastic problems on polytope
meshes.
Comput. Methods Appl. Mech. Engrg. \textbf{295}, 327--346 (2015).

\bibitem{BBPS2014} 
Benedetto, M.~F., Berrone, S., Pieraccini, S., Scial\`o, S.:
The virtual element method for discrete fracture network simulations.
Comput. Methods Appl. Mech. Engrg. \textbf{280}, 135--156 (2014).

\bibitem{BDMR95} 
Berm\'udez, A., Dur\'an, R., Muschietti, M.A., Rodr\'iguez, R.,
Solomin, J.:
Finite element vibration analysis of fluid-solid systems without
spurious modes.
SIAM J. Numer. Anal. \textbf{32}, 1280--1295 (1995).

\bibitem{BDRS} 
Berm\'udez, A., Dur\'an, R., Rodr\'iguez, R., Solomin, J.:
Finite element analysis of a quadratic eigenvalue problem arising in
dissipative acoustics.
SIAM J. Numer. Anal. \textbf{38}, 267--291 (2000).

\bibitem{BGHRS2008} 
Berm\'udez, A., Gamallo, P., Hervella-Nieto, L., Rodr\'iguez, R.,
Santamarina, D.:
Fluid-structure acoustic interaction.
In: Marburg, S., Nolte, B. (eds.)
Computational Acoustics of Noise Propagation in Fluids. Finite and
Boundary Element Methods.
Springer, Chap. 9, pp. 253--286 (2008).

\bibitem{B96} 
Berm\'udez, A., Hervella-Nieto, L., Rodr\'iguez, R.:
Finite element computation of three-dimensional elastoacoustic
vibrations. 
J. Sound Vibration, \textbf{219}, 279--306 (1999).

\bibitem{BHPR2001} 
Berm\'udez, A., Rodr\'iguez, R.:
Finite element computation of the vibration modes of a fluid-solid
system.
Comput. Methods Appl. Mech. Engrg. \textbf{119}, 355--370 (1994).

\bibitem{Boffi} 
Boffi, D.:
Finite element approximation of eigenvalue problems.
Acta Numerica, \textbf{19}, 1--120 (2010).

\bibitem{BGG2012} 
Boffi, D., Gardini, F., Gastaldi, L.:
Some remarks on eigenvalue approximation by finite elements.
In: Frontiers in Numerical Analysis--Durham 2010. 
Lect. Notes Comput. Sci. Eng. 
\textbf{85}, Springer, Heidelberg, 1--77 (2012).

\bibitem{BS-2008} 
Brenner, S.~C., Scott, R.~L.:
The Mathematical Theory of Finite Element Methods.
Springer, New York (2008).

\bibitem{ultimo} 
Brezzi, F., Falk, R.~S., Marini, L.~D.:
Basic principles of mixed virtual element methods.
ESAIM Math. Model. Numer. Anal. \textbf{48}, 1227--1240 (2014).

\bibitem{BM12} 
Brezzi, F., Marini, L.~D.:
Virtual elements for plate bending problems.
Comput. Methods Appl. Mech. Engrg. \textbf{253}, 455--462 (2012).

\bibitem{CGH14} 
Cangiani, A., Georgoulis, E.~H., Houston, P.:
$hp$-version discontinuous Galerkin methods on polygonal and polyhedral
meshes.
Math. Models Methods Appl. Sci. \textbf{24}, 2009--2041 (2014).

\bibitem{ciarlet}
Ciarlet, P.~G.:
The Finite Element Method for Elliptic Problems.
SIAM, Philadelphia (2002).

\bibitem{Raviart1} 
Descloux, J., Nassif, N., Rappaz, J.:
On spectral approximation. Part 1: The problem of convergence.
RAIRO Anal. Num\'er. \textbf{12}, 97--112 (1978).

\bibitem{Raviart2} 
Descloux, J., Nassif, N., Rappaz, J.:
On spectral approximation. Part 2: Error estimates for the Galerkin
method.
RAIRO Anal. Num\'er. \textbf{12}, 113--119 (1978).

\bibitem{DPECMAME2015} 
Di Pietro, D., Ern, A.:
A hybrid high-order locking-free method for linear elasticity on general
meshes.
Comput. Methods Appl. Mech. Eng. \textbf{283}, 1--21 (2015).

\bibitem{DPECRAS2015} 
Di Pietro, D., Ern, A.:
Hybrid high-order methods for variable-diffusion problems on general
meshes.
C. R. Acad. Sci., Paris I, \textbf{353}, 31--34 (2015).

\bibitem{Paulino-VEM}
Gain, A.~L., Talischi, C., Paulino, G.H.:
On the virtual element method for three-dimensional linear elasticity
problems on arbitrary polyhedral meshes.
Comput. Methods Appl. Mech. Engrg. \textbf{282}, 132--160 (2014).

\bibitem{DI2001} 
Hamdi, M., Ousset, Y.and Verchery, G.:
A displacement method for the analysis of vibrations of coupled
fluid-structure systems.
Internat. J. Numer. Methods Engrg. \textbf{13}, 139--150 (1978).

\bibitem{EM77} 
Kiefling, L., Feng, G.~C.:
Fluid-structure finite element vibrational analysis.
AIAA J. \textbf{14}, 199--203 (1976).

\bibitem{LMR} 
Lovadina, C., Mora, D., Rodr\'iguez, R.:
Approximation of the buckling problem for Reissner-Mindlin plates.
SIAM J. Numer. Anal. \textbf{48}, 603--632 (2010).

\bibitem{MRR2015} 
Mora, D., Rivera, G., Rodr\'iguez, R.:
A virtual element method for the Steklov eigenvalue problem.
Math. Models Methods Appl. Sci. \textbf{25}, 1421--1445 (2015).

\bibitem{PPR15} 
Perugia, I., Pietra, P., Russo, A.:
A plane wave virtual element method for the Helmholtz problem.
Preprint arXiv:1505.04965 [math.NA](2015).

\bibitem{HST2000} 
Rodr\'iguez, R., Solomin, J.:
The order of convergence of eigenfrequencies in
finite element approximations of fluid-structure interaction problems.
Math. Comp. \textbf{65}, 1463--1475 (1996). 

\bibitem{ST04} 
Sukumar, N., Tabarraei, A.:
Conforming polygonal finite elements.
Internat. J. Numer. Methods Engrg. \textbf{61}, 2045--2066 (2004).

\bibitem{TPPM10}
Talischi, C., Paulino, G.~H., Pereira, A., Menezes, I.~F.~M.:
Polygonal finite elements for topology optimization: A unifying
paradigm.
Internat. J. Numer. Methods Engrg. \textbf{82}, 671--698 (2010).

\bibitem{I98} 
Zienkiewicz, O.~C., Taylor, R.~L.:
The Finite Element Method.
Vol. 2, McGraw-Hill, London, 1991.
\end{thebibliography}

\end{document}